\documentclass{amsart}
\usepackage{epsfig}
\usepackage{url}

\newtheorem{theorem}{Theorem}[section]
\newtheorem{lemma}[theorem]{Lemma}
\newtheorem{corollary}[theorem]{Corollary}

\theoremstyle{definition}

\newcommand{\tf}{\tfrac}
\theoremstyle{remark}

\begin{document}

\title[Asymptotics and exact formulas for Zagier polynomials]{Asymptotics and exact formulas for Zagier polynomials} 

\author[A. Dixit]{Atul Dixit}
\address{Department of Mathematics,
Tulane University, New Orleans, LA 70118, USA}
\email{adixit@tulane.edu}
\curraddr{Department of Mathematics, Indian Institute of Technology Gandhinagar, 382355, Gujarat, India}
\email{adixit@iitgn.ac.in}

\author[M. L. Glasser]{M. Lawrence Glasser}
\address{Department of Physics, Clarkson University, Potsdam, NY 13369-5820, USA}
\email{laryg@clarkson.edu}

%\author[K. Mahlburg]{Karl Mahlburg}
%\address{Department of Mathematics, Louisiana State Univeristy, Baton Rouge, 
%LA 70803}
%\email{mahlburg@math.lsu.edu}

\author[V. H. Moll]{Victor H. Moll}
\address{Department of Mathematics,
Tulane University, New Orleans, LA 70118, USA}
\email{vhm@tulane.edu}

\author[C. Vignat]{Christophe Vignat}
\address{Department of Mathematics,
Tulane University, New Orleans, LA 70118, USA and 
L.S.S. Supelec, Universite d'Orsay, France}
\email{cvignat@tulane.edu}

%    General info

\date{\today}

\thanks{None of the authors have any competing interests in the manuscript.}

\begin{abstract}
In 1998 Don Zagier introduced the modified Bernoulli numbers $B_{n}^{*}$ and showed that they 
satisfy amusing variants of some properties of Bernoulli numbers. In particular, 
he studied the asymptotic behavior of $B_{2n}^{*}$, and also obtained an exact formula for them, the motivation for which came from 
the representation of $B_{2n}$ in terms of the Riemann zeta function $\zeta(2n)$. The 
modified Bernoulli numbers were recently generalized to Zagier polynomials $B_{n}^{*}(x)$. 
For $0<x<1$, an exact formula for 
$B_{2n}^{*}(x)$ involving infinite series of Bessel function of the second kind and Chebyshev polynomials, 
that yields Zagier's formula in a limiting case, is established here. Such series arise in diffraction theory. 
An analogous formula for $B_{2n+1}^{*}(x)$ 
is also presented. The $6$-periodicity of $B_{2n+1}^{*}$ is deduced as a limiting case of it. These formulas are reminiscent of the Fourier expansions of Bernoulli polynomials. Some new results, 
for example, the one yielding the derivative of the Bessel function of the first kind with respect to its order as 
the Fourier coefficient of a function involving Chebyshev polynomials, are obtained in the course of proving these
 exact formulas. The asymptotic behavior of 
Zagier polynomials is also derived from them. Finally, a Zagier-type exact formula is obtained for $B_{2n}^{*}\left(-\frac{3}{2}\right)+B_{2n}^{*}$.
\end{abstract}

\maketitle

\newcommand{\ba}{\begin{eqnarray}}
\newcommand{\ea}{\end{eqnarray}}
\newcommand{\ift}{\int_{0}^{\infty}}
\newcommand{\nn}{\nonumber}
\newcommand{\no}{\noindent}
\newcommand{\lf}{\left\lfloor}
\newcommand{\rf}{\right\rfloor}
\newcommand{\realpart}{\mathop{\rm Re}\nolimits}
\newcommand{\imagpart}{\mathop{\rm Im}\nolimits}
\newcommand{\pFq}[5]{\ensuremath{{}_{#1}F_{#2} \left( \genfrac{}{}{0pt}{}{#3}
{#4} \bigg| {#5} \right)}}

\newtheorem{Definition}{\bf Definition}[section]
\newtheorem{Thm}[Definition]{\bf Theorem} 
\newtheorem{Example}[Definition]{\bf Example} 
\newtheorem{Lem}[Definition]{\bf Lemma} 
\newtheorem{Cor}[Definition]{\bf Corollary} 
\newtheorem{Prop}[Definition]{\bf Proposition} 
\numberwithin{equation}{section}

\section{Introduction}
\label{sec-intro} 
\setcounter{equation}{0}
The Bernoulli polynomials $B_{n}(x)$ are defined by means of the generating function \cite[p.~3]{temme-1996a}
\begin{equation*}
\frac{ze^{xz}}{e^z-1}=\sum_{n=0}^{\infty}\frac{B_n(x)z^n}{n!}\hspace{8mm}(|z|<2\pi).
\end{equation*}
In \cite{zagier-1998a}, Don Zagier introduced the modified Bernoulli numbers 
\begin{equation}
B_{n}^{*} := \sum_{r=0}^{n} \binom{n+r}{2r} \frac{B_{r}}{n+r} \hspace{5mm} (n\in\mathbb{N}),
\label{mod-ber-1}
\end{equation}
\noindent
where $B_n:=B_n(0)$ are the Bernoulli numbers. 
In this paper, he proved three remarkable results for the sequence $\{B_{n}^{*}\}$: \\

\noindent
(A) The value of $B_{n}^{*}$ for $n$ odd is $6$-periodic: more precisely, it is 
given by 

\begin{center}
\begin{tabular}{c|c c c c c c} 
$n \bmod 12$ & $1$ & $3$ & $5$ & $7$ & $9$ & $11$ \\
\hline 
$B_{n}^{*}$ & $3/4$ & $- 1/4$ & $-1/4$ & $1/4$ & $1/4$ & $-3/4$ 
\end{tabular}.
\end{center}

\smallskip

\noindent
(B) The fractional part of the number $\tilde{B}_{2n}:= 4nB_{2n}^{*}-B_{2n}$ satisfies
\begin{equation*}
\tilde{B}_{2n} \equiv \sum_{\begin{stackrel}{(p+1)|2n} {p \text{ prime} }
\end{stackrel}}
\frac{1}{p} \bmod 1 \quad (n\in\mathbb{N}). 
\end{equation*}

\smallskip

\noindent
(C) Let $n\in\mathbb{N}$. Then $B_{2n}^{*}$ is asymptotically equal to 
$(-1)^{n-1}(2 \pi)^{-2n} (2n-1)!$ for $n$ large, and is 
given much more precisely by the approximation 
\begin{equation}\label{zagier-asym}
B_{2n}^{*} \approx (-1)^{n} \pi Y_{2n}(4 \pi) \quad ( n \to \infty),
\end{equation}
\noindent
where $Y_{n}(z)$ denotes the Bessel function of the second kind of integer order $n$ defined by \cite[p.~64]{watson-1966a}
\begin{equation}\label{bessely-int}
Y_{n}(z)=\lim_{\nu\to n}Y_{\nu}(z),
\end{equation}
where $Y_{\nu}(z)$ is the Bessel function of the second kind of non-integer order $\nu$ defined by
\begin{equation*}
Y_{\nu}(z)=\frac{J_{\nu}(z)\cos\left(\nu\pi\right)-J_{-\nu}(z)}{\sin\left(\nu\pi\right)},
\end{equation*}
with $J_{\nu}(z)$ being the Bessel function of the first kind \cite[p.~40]{watson-1966a}
\begin{equation}\label{besselj-int}
J_{\nu}(z)=\sum_{m=0}^{\infty}\frac{(-1)^m(z/2)^{2m+\nu}}{m!\Gamma(m+1+\nu)}.
\end{equation}
The sign $\approx$ in \eqref{zagier-asym} means that as $n\to\infty$, the relative error 
between the two sides decays more rapidly than any polynomial power $n^{-k}$. In particular, it yields the asymptotic 
\begin{equation}
\label{zag-asym}
B_{2n}^{*} \sim (-1)^{n} \pi Y_{2n}(4 \pi) \quad ( n \to \infty).
\end{equation}

The result on the $6$-periodicity of 
$B_{n}^{*}$ for odd $n$ first arose in Zagier's work \cite{zagier-1990b}, where he obtained a new proof, based on the theory of
periods of modular forms, of the Eichler-Selberg trace formula for the trace of Hecke operator $T_{\ell}$ acting 
on modular forms on $\text{SL}_2{(\mathbb{Z})}$. The method of the proof gave a formula for these traces in a form 
different than the usual, and involved Bernoulli numbers. The special case $\ell=1$ gave the dimension of
 $M_{k}\left(\text{SL}_2{\left(\mathbb{Z}\right)}\right)$ ($k$ even) in terms of $B_{k-1}^{*}$ and the equality of this formula
 with the standard dimension formula required the $6$-periodicity to hold.

The result in $(A)$ was extended in \cite{dixit-2014a} to the so-called 
Zagier polynomials 
\begin{equation*}
B_{n}^{*}(x) = \sum_{r=0}^{n} \binom{n+r}{2r} \frac{B_{r}(x)}{n+r} \hspace{5mm} (n>0),
\end{equation*}
\noindent
where $B_{n}^{*}(0) = B_{n}^{*}$. It should be mentioned here that these polynomials were
briefly studied by Zagier himself, and appear in an exercise in \cite[p.~122, Exercise 20]{cohen-2008a}. Specifically, he
 obtained some results associated with their generating function $\sum_{n=1}^{\infty}B_{n}^{*}(x)z^{n}$,
 viewed as a formal power series. However, an explicit formula for this generating function in terms
 of the digamma function was first obtained in \cite[Theorem 3.1]{dixit-2014a}.

As shown in \cite[Theorem 1.2]{dixit-2014a},  
$\{B_{2n+1}^{*}(j)\}_{n=0}^{\infty}$ is periodic and non-constant precisely when $j \in 
\{ -3, \, -2, \, -1, \, 0\}$. The period is $6$ for $j=-3, \, 0$, and $2$ 
for $j=-2, \, -1$. Moreover $\{B_{2n+1}^{*}(j)+n\}_{n=0}^{\infty}$ is periodic if and only if  
$j=-4$. Equivalently, using symmetry result for Zagier polynomials (see \eqref{refsym} below), $\{-B_{2n+1}^{*}(j)+n\}_{n=0}^{\infty}$ is periodic if and only if  
$j=1$. The only other value of $j$ that yields a periodic example is $j=-3/2$, 
in which case we get the vanishing sequence since 
$B_{2n+1}^{*}(-3/2) \equiv 0$. Similar periodicity results hold for the sequence 
$\{B_{2n}^{*}(-1-j) - B_{2n}^{*}(-1)\}_{n=1}^{\infty}$, see \cite{dixit-2014a}.  

\smallskip

The second paper in this series \cite{dixit-2014b} was inspired from property (B) above, and the  
arithmetic nature of Zagier polynomials with integer arguments was studied there.  Let $\alpha_{n}$ and $\gamma_{n,j}$ respectively be the denominators of $B_{n}^{*}$  
and $B_{n}^{*}(j)$ in reduced forms, where $j \in 
\mathbb{Z}$ and $n \in \mathbb{N}$. Again, two different behaviors, one for $n$ even 
and another for $n$ odd, are observed. The fact that $\alpha_{2n+1}=4$ for all $n\in\mathbb{N}$ was proved in 
\cite{dixit-2014a}. It was also conjectured there that $4$ divides $\alpha_{2n}$ for all $n\in\mathbb{N}$. This
conjecture was established in \cite{dixit-2014b} in the form of the following theorem, which 
also implies that $\gamma_{n,j}$ is independent of $j$.
\begin{theorem}
\label{thm:2adic}
  Let $p$ be a prime and let $\nu_{p}(\ell)$ denote the $p$-adic valuation of $\ell$ (that is, the highest power of $p$ that divides $\ell$). For $n\in\mathbb{N}$,
  \begin{align*}
    \nu_2(\alpha_n) = -\nu_2(B^*_n) = 2 + \nu_2(n)
    - \begin{cases} 1 & \text{if $n \equiv 6 \bmod 12$}, \\
      2 & \text{if $n \equiv 0 \bmod 12$}, \\
      0 & \text{otherwise.} \end{cases}
  \end{align*}
%In particular, 
%$B_{n}^{*}$, the denominator of $\alpha_{n}$, is divisible by $4$.
\end{theorem}
%The conjecture is related to the sequence $\{ z_{n} \}$ defined by the 
%recurrence 
%\begin{equation}
%z_{n} = \binom{2n}{n} - \sum_{k=1}^{n-1} \binom{2n}{n+k} z_{k} - B_{2n}
%\end{equation}
%\noindent
%where $B_{2n}$ is the Bernoulli numbers and $z_{1} = 11/6$. It is 
%conjectured that the 
%$2$-adic valuation of $\{ z_{n} \}$ forms a periodic sequence of period $6$. 
%If true, this would imply  that $4$ divides $\alpha_{2n}$. 

\smallskip 
The result in (C) motivates the present paper. Zagier \cite{zagier-1998a} obtained \eqref{zag-asym}, analogous to the asymptotic formula \cite[p.~267]{apostol-1998a}
\begin{equation*}
B_{2n}\sim2(2\pi)^{-2n}(-1)^{n+1}(2n)!\hspace{10mm} (n\to\infty)
\end{equation*}
satisfied by the Bernoulli numbers.
The above asymptotic formula for Bernoulli numbers can be replaced by the following exact formula \cite[p.~266]{apostol-1998a}, \cite[p.~5, Equation (1.14)]{temme-1996a}
\begin{equation}\label{berexact}
B_{2n}=\frac{2(-1)^{n+1}(2n)!\zeta(2n)}{(2\pi)^{2n}},
\end{equation}
where $\zeta(s)$ is the Riemann zeta function. Zagier, in looking for a corresponding exact formula
 for $B_{2n}^{*}$, obtained the following beautiful result \cite{zagier-1998a}
\begin{eqnarray}
B_{2n}^{*}  & =  & -n+
\sum_{\ell=1}^{\infty} \left( (-1)^{n} \pi Y_{2n}(4 \pi \ell) + 
\frac{1}{2 \sqrt{\ell}} \right) - \frac{1}{2} \zeta \left( \frac{1}{2} \right) 
\label{zagier-sum} \\
& & + \sum_{\ell=1}^{\infty} \frac{1}{\sqrt{\ell(\ell+4)}} 
\left( \frac{\sqrt{\ell+4} - \sqrt{\ell}}{2} \right)^{4n}.
\nonumber
\end{eqnarray}
\no
The heuristic behind the discovery of this formula is interesting, and the interested reader is referred
to \cite{zagier-1998a}. The proof of \eqref{zagier-sum} is pretty and 
involves meticulous manipulations of the associated infinite series and integrals.

Note that the companion formula 
\begin{equation}\label{comp-1}
B_{2n+1}^{*} = \frac{1}{4}\left(\frac{-4}{2n+1}\right)+\frac{1}{2}\left(\frac{-3}{2n+1}\right),
\end{equation}
where $\left(\frac{a}{n}\right)$ is the Jacobi symbol, is elementary and has also been established
 in \cite{zagier-1998a}. It can also be rephrased in the form \cite[Corollary 10.6]{dixit-2014a}
\begin{equation}\label{comp-2}
B_{2n+1}^{*} = \frac{(-1)^{n}}{4} + \frac{1}{\sqrt{3}} 
\sin \left( \frac{(2n+1)\pi}{3} \right).
\end{equation}
\no
For $0\leq x\leq 1$ and $n\geq 1$, the Fourier expansion of the even-indexed Bernoulli polynomials is given by 
\cite[p. 5]{temme-1996a} 
\begin{equation}
B_{2n}(x) = 2(-1)^{n+1} (2n)! \sum_{m=1}^{\infty} \frac{\cos 2 \pi m x}
{(2 \pi m)^{2n}},
\label{ber-series1}
\end{equation}
which gives \eqref{berexact} as a special case when $x=0$. Similarly, for $0\leq x\leq 1$ when $n>0$, and for $0<x<1$ when $n=0$, the Fourier expansion of the odd-indexed Bernoulli polynomials is given by \cite[p. 5]{temme-1996a} \footnote{There is a typo in the version of this formula given in the book. The power of $-1$ there should be $n$ and not $n+1$.}
\begin{equation}
B_{2n+1}(x) = 2(-1)^{n+1} (2n+1)! \sum_{m=1}^{\infty} \frac{\sin 2 \pi m x}
{(2 \pi m)^{2n+1}}.
\label{ber-series2}
\end{equation}
The Fourier expansions in \eqref{ber-series1} and \eqref{ber-series2} now raise two natural questions: does there exist a generalization of \eqref{zagier-sum} 
for the Zagier polynomials $B_{2n}^{*}(x)$, and, is there an analogue of such a generalization for $B_{2n+1}^{*}(x)$?

The primary goal of this paper is to answer these two questions in the affirmative. The generalization of 
\eqref{zagier-sum} for the polynomial $B_{2n}^{*}(x)$ when $0 < x < 1$ is given first.

\begin{theorem}
\label{thm-main-0}
Let $0 < x < 1$ and $n \in \mathbb{N} $. Define
\begin{equation}\label{gynx}
g(y, r, x):=\frac{(y+1+x - \sqrt{(y-1+x)(y+3+x) } \,\, )^{2r}}{\sqrt{(y-1+x)(y+3+x)}}.
\end{equation}
Let $Y_{n}(z)$ be defined in \eqref{bessely-int}, and denote by $U_{n}(x)$ the Chebyshev polynomial of the
second kind. Then,
\begin{eqnarray*}
B_{2n}^{*}(x) &  = & (-1)^{n } \pi\sum_{m=1}^{\infty} Y_{2n}(4 \pi m) \cos(2 \pi m x) \\
& & + \frac{1}{4} \left( U_{2n-1} \left( \frac{x+1}{2} \right) + U_{2n-1} \left( \frac{x}{2} \right) +
U_{2n-1} \left( \frac{x-1}{2} \right) + U_{2n-1} \left( \frac{x-2}{2} \right)  \right) \\
& & + \frac{1}{2^{2n+1}} \left(\sum_{m=1}^{\infty} g(m, n, x)+\sum_{m=1}^{\infty}g(m, n, 1-x)\right) .
 \end{eqnarray*}
 \end{theorem}
 
% The asymptotic behavior of $B_{2n}^{*}(x)$ is described next. 
 
% \begin{corollary}
 %In the expression for $B_{2n}^{*}(x)$ given above, the first term of the infinite series is the dominant term as $n \to \infty$. In particular, 
 %\begin{equation}
 %B_{2n}^{*}(x) \sim (-1)^{n} \pi Y_{2n}(4 \pi) \cos( 2 \pi x)
 %\end{equation}
 %\noindent
 %as $n \to \infty$.
 %\end{corollary}

The analogue of Theorem \ref{thm-main-0} for $B_{2n+1}^{*}(x)$ is given next.
\begin{theorem}
\label{thm-main-1}
Let $0 < x < 1$ and $n \in \mathbb{N} $.  Let $g(y, r, x)$, $Y_n(z)$ and $U_n(x)$ be defined as before. Then,
\begin{eqnarray*}
B_{2n+1}^{*}(x) &  = & (-1)^{n } \pi\sum_{m=1}^{\infty} Y_{2n+1}(4 \pi m) \sin(2 \pi m x) \\
& & + \frac{1}{4} \left( U_{2n} \left( \frac{x+1}{2} \right) + U_{2n} \left( \frac{x}{2} \right) +
U_{2n} \left( \frac{x-1}{2} \right) + U_{2n} \left( \frac{x-2}{2} \right)  \right) \\
& & + \frac{1}{2^{2n+2}} \left(\sum_{m=1}^{\infty} g\left(m, n+\tfrac{1}{2},x\right)-\sum_{m=1}^{\infty}g\left(m, n+\tfrac{1}{2},1-x\right)\right).
 \end{eqnarray*}
 \end{theorem}%
%Curious special cases of the Theorems \ref{thm-main-0} and \ref{thm-main-1} are derived in Sections \ref{sec-recovery}  and \ref{sec-periodicity}. 
\textbf{Remark 1.} The fact that the Bessel function series appearing in Theorems \ref{thm-main-0} and \ref{thm-main-1} 
converge conditionally can be proved as follows. The asymptotic expansion of $Y_{\nu}(z)$  as $|z|\to\infty$ (see \eqref{asymbess1} below) implies that 
as $m\to\infty$,
\begin{equation}
(-1)^n\pi Y_{2n}(4\pi m)\sim-\frac{1}{2\sqrt{m}}.\label{Y-diverge}
\end{equation}
Now it is well-known \cite[p.~257]{apostol-1998a} that the series $\sum_{m=1}^{\infty} e^{2\pi imx}m^{-s}$ converges conditionally for Re$(s)>0$, 
which means that each of the series $\sum_{m=1}^{\infty}\frac{\cos(2\pi mx)}{\sqrt{m}}$ and $\sum_{m=1}^{\infty}\frac{\sin(2\pi mx)}{\sqrt{m}}$ 
also converges conditionally, and hence the series in Theorems \ref{thm-main-0} and \ref{thm-main-1} as well.

\textbf{Remark 2.} Theorems \ref{thm-main-0} and \ref{thm-main-1} combined with the result \cite[Lemma 10.2]{dixit-2014a}
\begin{equation*}
B_{n}^{*}(x+1) = B_{n}^{*}(x) + \frac{1}{2}U_{n-1} \left( \frac{x}{2} + 1
\right),
\end{equation*}
or more generally with
\begin{equation}\label{lem10.2}
B_{n}^{*}(x+k) = B_{n}^{*}(x) + \frac{1}{2} \sum_{j=1}^{k} 
U_{n-1} \left( \frac{x+j-1}{2} + 1 \right)\hspace{7mm} (k\in\mathbb{Z}),
\end{equation}
give exact formulas for $B_{n}^{*}(x)$ for any non-integer real values of $x$.

In a completely different context of diffraction theory, V.~Twersky \cite[Equations (40), (41)]{twersky-1961c} (see also C.~M.~Linton \cite{linton-2006b}, \cite[Equations (47), (49)]{linton-2006a}) has obtained one of the intermediate results in the proof of our Theorem \ref{thm-main-0}, namely \eqref{finalaxn} (and likewise the corresponding equation occuring in the proof of our Theorem \ref{thm-main-1}), albeit these are expressed in forms much different than ours,  especially since they are phrased using the terminology of diffraction grating. However, our proofs of these intermediate results are new and completely different from his. Our proofs also give new and important results along the way, for e.g., Lemmas \ref{lemma-two4} and \ref{lemma-two5}. 

As hinted above, infinite series involving Bessel functions and trigonometric functions are often encountered in studies on the theory of diffraction 
\cite{asatryan}, \cite{hargreaves-1918a}, \cite{jackson-1904a}, \cite[Appendix D]{macdonald-2013a}, \cite{mcphedran-1983a}. Indeed, the series in the above theorems have arisen \cite{twersky-1956a}, \cite{twersky-1959b} 
in the analysis of the scattering of a plane wave on a diffraction grating with an arbitrary angle of incidence. The Schl\"omilch or Schl\"omilch type series, such as the one in Theorems \ref{thm-main-0} and \ref{thm-main-1}, or the ones studied in \cite{asatryan}, are conditionally convergent and converge extremely slowly. Their partial sums are highly oscillatory. Thus the advantage of alternate representations for them, such as the ones in Theorems \ref{thm-main-0} and \ref{thm-main-1}, is that they are useful for their fast computation.

W. v.~Ignatowsky \cite[Section 6]{ignatowsky-1915a} studied the series $\sum_{m=1}^{\infty}Y_{2n}(mD)$, where $n\in\mathbb{N}, D>0$ and $D$ is not an integral multiple of $2\pi$. His case obviously does not cover the series in Zagier's formula \eqref{zagier-sum}. However, we show that Zagier's formula can be obtained from Theorem \ref{thm-main-0} in the limiting case $x\to 1$, the proof of which is interesting in itself. Also, it is shown that the limiting case $x\to 0$ of Theorem \ref{thm-main-1} gives the curious $6$-periodicity of $B_{2n+1}^{*}$ mentioned in property (A), and which is equivalent to \eqref{comp-1} and \eqref{comp-2}. 

K.~Dilcher \cite[Corollary 1]{dilcher-1987a} showed that the sequence of Bernoulli polynomials 
converges uniformly on compact subsets of $\mathbb{C}$ to the sine or cosine functions. 
This implies, in particular, the asymptotic formulas
\begin{align*}
B_{2n}(x)&\sim 2(2\pi)^{-2n}(2n)!(-1)^{n-1}\cos(2\pi x),\\
B_{2n+1}(x)&\sim 2(2\pi)^{-(2n+1)}(2n+1)!(-1)^{n-1}\sin(2\pi x),
\end{align*}
for real $x$ as $n\to\infty$.

As an application of our Theorems \ref{thm-main-0} and \ref{thm-main-1}, the following asymptotic relations for 
Zagier polynomials are obtained here.
\begin{corollary}\label{zagier-asymp}
Let $0<x<1$. For $x\neq\frac{1}{4},\frac{3}{4}$, as $n \to \infty$,
\begin{align}\label{zagier-even-asymp}
 B_{2n}^{*}(x) &\sim (-1)^{n} \pi Y_{2n}(4 \pi) \cos( 2 \pi x),
 \end{align} 
and
\begin{align*}
 B_{2n}^{*}\left(\frac{1}{4}\right) &\sim (-1)^{n+1} \pi Y_{2n}(8 \pi)\sim B_{2n}^{*}\left(\frac{3}{4}\right).
 \end{align*} 
Also for $x\neq\frac{1}{2}$, as $n\to\infty$,
\begin{align*}
 B_{2n+1}^{*}(x) &\sim (-1)^{n} \pi Y_{2n+1}(4 \pi) \sin( 2 \pi x).
 \end{align*}
\end{corollary}
Lastly, we obtain a Zagier-type formula linking an infinite series involving $Y_{2n}(8\pi m)$ with $B_{2n}^{*}\left(-\frac{3}{2}\right)$ and $B_{2n}^{*}$. This formula is similar in flavor to Zagier's formula \eqref{zagier-sum}. The genesis of this formula is explained in Section \ref{op}.
\begin{theorem}\label{zagier-type-thm}
Let $Y_{n}(z)$, $U_n(x)$ be defined as before. The identity
 {\allowdisplaybreaks\begin{align*}
B_{2n}^{*}\left(-\frac{3}{2}\right)+B_{2n}^{*}&=2\sum_{m=1}^{\infty}\left((-1)^{n}\pi Y_{2n}(8\pi m)+\frac{1}{2\sqrt{2m}}\right)\nonumber\\
&\quad-n-\frac{1}{2}\left(U_{2n-1}\left(\frac{1}{4}\right)+U_{2n-1}\left(\frac{3}{4}\right)\right)-\frac{1}{\sqrt{2}}\zeta\left(\frac{1}{2}\right)\nonumber\\
&\quad+\frac{1}{2^{4n-1}}\sum_{m=1}^{\infty}\frac{\left(m+4-\sqrt{m(m+8)}\right)^{2n}}{\sqrt{m(m+8)}}
\end{align*}}
is true for all $n\in\mathbb{N}$.
\end{theorem}
%, for example Lemma \ref{lemma-two4}.
This paper is organized as follows. The preliminary results are collected in Section \ref{prelim}. Theorem \ref{thm-main-0} is proved in Section \ref{sec-proof-main}. We do not give the proof of Theorem \ref{thm-main-1} since the approach is similar to that of Theorem \ref{thm-main-0}. Section \ref{sec-recovery} is devoted to deriving Zagier's formula \eqref{zagier-sum} as a special case of Theorem \ref{thm-main-0}. Similarly, Section \ref{sec-periodicity} contains proof of the $6$-periodicity of $B_{2n+1}^{*}$ resulting from Theorem \ref{thm-main-1}. The asymptotic properties of Zagier polynomials are proved in Section \ref{asymptotics}. The Zagier-type exact formula for $B_{2n}^{*}\left(-\tfrac{3}{2}\right)+B_{2n}^{*}$ is derived in Section \ref{zagier-type}. Finally the paper concludes with three open problems discussed in Section \ref{op}.

\section{Preliminaries}\label{prelim}
The Bessel functions $J_{\nu}(z)$, and $Y_{\nu}(z)$ have the following asymptotic expansions for $|z|\to\infty$ and  $|\arg z|<\pi$ \cite[p.~199]{watson-1966a}:
 {\allowdisplaybreaks\begin{align}
J_{\nu}(z)&\sim \left(\frac{2}{\pi z}\right)^{\tfrac{1}{2}}\bigg(\cos w\sum_{n=0}^{\infty}\frac{(-1)^n(\nu, 2n)}{(2z)^{2n}} -\sin w\sum_{n=0}^{\infty}\frac{(-1)^n(\nu, 2n+1)}{(2z)^{2n+1}}\bigg),\nonumber\\
Y_{\nu}(z)&\sim \left(\frac{2}{\pi z}\right)^{\tf12}\bigg(\sin w\sum_{n=0}^{\infty}\frac{(-1)^n(\nu, 2n)}{(2z)^{2n}}+\cos w\sum_{n=0}^{\infty}\frac{(-1)^n(\nu, 2n+1)}{(2z)^{2n+1}}\bigg).\label{asymbess1}
\end{align}}
Here $w=z-\tfrac{1}{2}\pi \nu-\tfrac{1}{4}\pi$ and $(\nu,n)=\frac{\Gamma(\nu+n+1/2)}{\Gamma(n+1)\Gamma(\nu-n+1/2)}$.

Also, if $\nu\to\infty$ through positive real values, then the asymptotic expansions of $J_{\nu}(z)$ and $Y_{\nu}(z)$ for a non-zero fixed $z$ are given by \cite[p.~231]{olver-2010a}
\begin{align}
J_{\nu}(z)&\sim \frac{1}{\sqrt{2\pi\nu}}\left(\frac{ez}{2\nu}\right)^{\nu},\nonumber\\
Y_{\nu}(z)&\sim -\sqrt{\frac{2}{\pi\nu}}\left(\frac{ez}{2\nu}\right)^{-\nu}.\label{yorder-asymp}
\end{align}
The property \cite[p.~222, 10.4.1]{olver-2010a}
\begin{equation*}
J_{-n}(z)=(-1)^nJ_{n}(z)
\end{equation*}
for $n\in\mathbb{N}$ is used throughout the paper without mention. So will be the facts 
\begin{equation*}
J_{0}(0)=1, \hspace{4mm} J_{\nu}(0)=0,
\end{equation*}
for Re$(\nu)>0$.

The Chebyshev polynomials of the first and second kinds are respectively defined for $n\geq 0$ by the Binet formulas
\begin{align}
T_{n}(x) &= \frac{( x + \sqrt{x^{2}-1} )^{n} + ( x - \sqrt{x^{2}-1})^{n}}{2},\nonumber\\
U_{n}(x) &= \frac{( x + \sqrt{x^{2}-1} )^{n+1} - ( x - \sqrt{x^{2}-1})^{n+1}}{2 \sqrt{x^{2}-1}}\label{chebyshev}.
\end{align}
They are alternatively given by
\begin{align*}
T_{n}(\cos \theta) = \cos (n \theta), \hspace{4mm}
U_{n} ( \cos \theta)  = \frac{\sin( (n+1) \theta)}{\sin \theta}.
\end{align*}
The Poisson summation formula \cite[p.~60-61]{titchmarsh-1948a} states that if $f(t)$ is continuous and of 
bounded variation on $[0,\infty)$, and if $\int_{0}^{\infty}f(t)\, dt$ exists, then
 \begin{equation}\label{poisson}
 f(0) + 2 \sum_{m=1}^{\infty} f(m) = 2 \int_{0}^{\infty} f(t) \, dt + 4 \sum_{m=1}^{\infty} \int_{0}^{\infty} f(t) \cos(2 \pi m t ) \, dt.
 \end{equation}

\section{Proof of Theorem \ref{thm-main-0}}
\label{sec-proof-main}
\setcounter{equation}{0}

The proof of \eqref{zagier-sum}, as given in \cite{zagier-1998a}, begins with the representation \eqref{mod-ber-1} and 
uses the fact that all odd-indexed Bernoulli numbers, except the first one, vanish and that the even-indexed ones, namely $B_{2r}$, 
can be expressed in terms of the Riemann zeta function $\zeta(2r)$. This procedure does not extend very well to the case of Bernoulli 
polynomials, simply because the odd-indexed Bernoulli polynomials do not vanish, and leads to two complicated 
terminating ${}_5F_{2}$ hypergeometric functions which do not seem to produce anything like \eqref{zagier-sum}. 

The idea is to start with the formula 
\begin{equation*}
2 B_{2n}^{*}(x) = \sum_{r=0}^{n} (-1)^{n+r} \binom{n+r}{2r} \frac{B_{2r}(x)}{n+r} + 
U_{2n-1} \left( \frac{x}{2} \right) + U_{2n-1} \left( \frac{x+1}{2} \right),
\end{equation*}
\noindent 
established in \cite[Theorem 10.1]{dixit-2014b}. Observe that only even-indexed Bernoulli polynomials appear in this 
representation.
% the Fourier expansion of Bernoulli polynomials \cite[p. 5]{temme-1996a} 
%\begin{equation}
%B_{2n}(x) = 2(-1)^{n+1} (2n)! \sum_{m=1}^{\infty} \frac{\cos 2 \pi m x}
%{(2 \pi m)^{2n}}.
%\label
%\end{equation}
%\noindent
Separating the term $r=0$ and then using \eqref{ber-series1} yields 
\begin{equation}
\label{formula-ber11}
B_{2n}^{*}(x) = \frac{(-1)^{n}}{2n} + A(n,x) + \frac{1}{2} \left( 
U_{2n-1} \left( \frac{x}{2} \right) + U_{2n-1} \left( \frac{x+1}{2} \right) \right)
\end{equation}
\noindent
with $A(n,x)$ defined by 
\begin{equation*}
A(n,x) = (-1)^{n+1} \sum_{m=1}^{\infty} \cos(2 \pi m x) \sum_{r=0}^{n-1}  \frac{(n+r)!}{(n-r-1)!} 
\frac{1}{(2 \pi m)^{2r+2}}.
\end{equation*}

Upon replacing $r$ by $n-1-r$ in the above sum, the function $A(n,x)$ can be written in the form
\begin{equation}
A(n,x) = (-1)^{n+1} \sum_{m=1}^{\infty} S_{2n}(4 \pi m) \cos(2 \pi m x)
\label{form-A1}
\end{equation}
\noindent
where $S_{n}(x)$ is the Schl\"{a}fli  polynomial  \cite[p.~285]{watson-1966a} defined by $S_{0}(z) = 0$ and 
\begin{equation*}
S_{n}(z) = \sum_{r=0}^{\tfrac{n-a}{2}} \frac{(n-r-1)!}{r!} \left( \frac{z}{2} \right)^{2r-n}
\end{equation*}
\noindent
with 
\begin{equation*}
a = \begin{cases}
2 & \quad \text{ for } n \text{ even}, \\
1 & \quad \text{ for } n \text{ odd}. 
\end{cases}
\end{equation*}

The proof is now broken down into a series of lemmas for an easy perusal. 
A new expression for $A(n,x)$ is first presented. 
\begin{lemma}\label{lemma-first}
The function $A(n,x)$ is given by 
{\allowdisplaybreaks\begin{multline*}
A(n,x)  =  (-1)^{n} \pi\sum_{m=1}^{\infty} Y_{2n}(4 \pi m) \cos(2 \pi m x) \\
 \quad  + (-1)^{n+1} \sum_{m=1}^{\infty} \left\{ 2 ( \gamma + \log(2 \pi m) ) J_{2n}(4 \pi m) + 
 P_{2n}(4 \pi m ) - 2 Q_{2n}(4 \pi m) \right\} \cos(2 \pi m x),
 \end{multline*}}
 \noindent 
 where $\gamma$ is Euler's constant,  $Y_{n}(x)$ and $J_{n}(x)$ are defined by \eqref{bessely-int} and \eqref{besselj-int} respectively, and the functions 
 $P_{n}, \, Q_{n}$ are given by \cite[p. 341]{watson-1966a} \footnote{We have used $P_n(z)$ and $Q_n(z)$, instead of the standard notation $T_n(z)$ and $U_n(z)$, in order to avoid any possible confusion with the notation for the Chebyshev polynomials of the first and second kind.}
 \begin{eqnarray*}
 P_{n}(z)  & = &    - \sum_{r \geq \tfrac{n}{2}}^{n-1} 
 \frac{ (n-r-1)!}{r!} \left( \frac{z}{2} \right)^{2r-n}   \\ 
 & & \quad \quad + 
 \sum_{\ell=0}^{\infty} (-1)^{\ell}  \left( \frac{z}{2} \right)^{n+ 2 \ell} 
 \frac{( \psi(n+ \ell+1) - \psi(\ell+1))}{\ell! \, (n+ \ell)!} 
 \end{eqnarray*}
 \noindent
 and 
 \begin{eqnarray*}
 Q_{n}(z) & = & \sum_{\ell=0}^{\infty} (-1)^{\ell} \left( \frac{z}{2} \right)^{n+ 2 \ell} 
 \frac{( \psi(n + \ell + 1) + \gamma )}{\ell! \, (n+ \ell)!},
 \end{eqnarray*}
where $\psi(z)$ is the logarithmic derivative of the gamma function $\Gamma(z)$.
 \end{lemma}
 \begin{proof}
 For $n \in \mathbb{N}$, the identity \cite[p. 340, 10.6]{watson-1966a} 
 \begin{eqnarray*}
 \pi Y_{n}(z) & = & - \sum_{r=0}^{n-1} \frac{(n-r-1)!}{r!} \left( \frac{z}{2} \right)^{2r-n} \label{form-Y1} \\
 & + & \sum_{\ell=0}^{\infty} (-1)^{\ell} \left( \frac{z}{2} \right)^{n+ 2 \ell} 
 \frac{ ( 2 \log (z/2) - \psi( \ell+1) - \psi(n+\ell+1)}{\ell! \, (n + \ell)!}
 \nonumber
 \end{eqnarray*}
\noindent
can be written in the form \cite[p. 340, 10.6(1)]{watson-1966a}
\begin{equation}
S_{n}(z) = - \pi Y_{n}(z) + 2 ( \gamma + \log(z/2) ) J_{n}(z) + P_{n}(z) - 2Q_{n}(z).
\label{form-S1} 
\end{equation}
\noindent
Now substitute \eqref{form-S1}, with $z=4\pi m$ and $n$ replaced by $2n$, in \eqref{form-A1} to produce 
the result. That the right-hand side can be written as the sum of two series follows from the first 
remark after the statement of Theorem \ref{thm-main-1}.
\end{proof}
The next lemma gives yet another expression for the function $A(n, x)$.
\begin{lemma}
\label{lemma-bessel-00}
The function $A(n,x)$ is given by 
\begin{eqnarray*}
A(n,x) & = & (-1)^{n} \pi\sum_{m=1}^{\infty} Y_{2n}(4 \pi m) \cos(2 \pi m x) \\
 & + & (-1)^{n+1}\sum_{m=1}^{\infty} \left(P_{2n}(4 \pi m) +  2\left.\frac{\partial}{\partial \nu}J_{\nu}(4 \pi m )\right|_{\nu=2n}\right)\cos( 2 \pi m x).
 \end{eqnarray*}
 \end{lemma}
 \begin{proof}
 The functions $P_{n}(z)$ and $Q_{n}(z)$ are also given by \cite[p. 344]{watson-1966a}
 \begin{eqnarray}
 P_{n}(z) & = & \sum_{k=1}^{\infty} \frac{1}{k} ( J_{n+2k}(z) - J_{n- 2k}(z) )  \label{form-PQ1}  \\
 Q_{n}(z) & = & J_{n}(z) \sum_{j=1}^{n} \frac{1}{j} + \sum_{k=1}^{\infty}  
 \frac{(-1)^{k} (n+2k)}{k(n+k)} J_{n+2k}(z). \label{form-PQ2}
 \end{eqnarray}
 Also, for a non-negative integer $\nu$, the derivative of $J_{\nu}(z)$ with respect to its order 
is given by \cite[p. 53, formula (39)]{luke-1969b}
 \begin{equation}\label{besselj-der}
 \frac{\partial}{\partial \nu} J_{\nu}(z) = \left( \log \left( \frac{z}{2} \right) - \psi(\nu+1) \right)  J_{\nu}(z) - 
 \sum_{k=1}^{\infty} \frac{(-1)^{k} (2k+ \nu)}{k(k+ \nu)} J_{2k + \nu}(z),
 \end{equation}
 \noindent
 Now \eqref{form-PQ2} and \eqref{besselj-der} together give
\begin{equation*}
Q_{2n}(4\pi m)=J_{2n}(4\pi m)\bigg(\sum_{j=1}^{2n}\frac{1}{j}+\log\left(2\pi m\right)-\psi(2n+1)\bigg)-\left.\frac{\partial}{\partial\nu}J_{\nu}(4\pi m)\right|_{\nu=2n}.
\end{equation*}
so that
\begin{align*}
&2\left(\gamma+\log\left(2\pi m\right)\right)J_{2n}(4\pi m)+P_{2n}(4\pi m)-2Q_{2n}(4\pi m)\nonumber\\
&=2\bigg(\gamma+\psi(2n+1)-\sum_{j=1}^{2n}\frac{1}{j}\bigg)J_{2n}(4\pi m)+P_{2n}(4\pi m)+2\left.\frac{\partial}{\partial\nu}J_{\nu}(4\pi m)\right|_{\nu=2n}.
\end{align*}
 The identity 
 \begin{equation*}
 \psi(2n+1) = -\gamma + \sum_{j=1}^{2n} \frac{1}{j}
 \end{equation*}
 \noindent
 that appears as entry $8.365.4$ in \cite{gradshteyn-2015a} and Lemma \ref{lemma-first} now complete the proof. 
 \end{proof}
 The next task is to find an almost closed-form expression for the second series in the 
above lemma. The following two new lemmas, interesting in their own right, show that the functions $P_{2n}(4\pi m)$ 
and $\left.\frac{\partial}{\partial \nu}J_{\nu}(4 \pi m )\right|_{\nu=2n}$ arise as 
Fourier coefficients in the Fourier expansions of some relatively simple functions.
\begin{lemma}
 \label{lemma-two4}
 For $n \in \mathbb{N}$ and $0 < x < 1$, the identity 
 \begin{align*}
&\sum_{m=1}^{\infty}P_{2n}(4\pi m)\cos(2\pi mx)\nonumber\\
&=\frac{1}{2n}+\frac{(-1)^n}{2\pi}\bigg\{\cos^{-1}\left(\frac{x}{2}\right)U_{2n-1}\left(\frac{x}{2}\right)+\cos^{-1}\left(\frac{x+1}{2}\right)U_{2n-1}\left(\frac{x+1}{2}\right)\nonumber\\
&\quad\quad\quad\quad\quad+\cos^{-1}\left(\frac{1-x}{2}\right)U_{2n-1}\left(\frac{1-x}{2}\right)+\cos^{-1}\left(\frac{2-x}{2}\right)U_{2n-1}\left(\frac{2-x}{2}\right)\bigg\}
\end{align*}
 holds.
 \end{lemma}
\begin{proof}
Let $f(x, n)$ denote the right-hand side of the above identity. Note that $f(1-x, n)=f(x, n)$. 
The periodization of $f(x, n)$, as a function of $x$, based on its values in $[0,1)$ makes it 
an even function of $x$. Hence its Fourier series is given by
\begin{equation*}
a_0+\sum_{m=1}^{\infty}a_m\cos(2\pi m x),
\end{equation*}
where
\begin{align*}
a_0=\int_{0}^{1}f(x, n)\, dx, \hspace{6mm} a_m=2\int_{0}^{1}f(x, n)\cos(2\pi m x)\, dx,
\end{align*}
for $m\geq 1$. These Fourier coefficients are now computed. Note that the change of 
variable $x=2\cos\theta$ yields 
\begin{equation*}
\int_{0}^{1}U_{2n-1}\left(\frac{x}{2}\right)\cos^{-1}\left(\frac{x}{2}\right)\, dx=2\int_{\pi/3}^{\pi/2}\theta\sin(2n\theta)\, d\theta.
\end{equation*}
Similarly, combining all such integrals, it is seen that
\begin{align*}
a_0&=\frac{1}{2n}+\frac{2(-1)^n}{\pi}\int_{0}^{\pi/2}\theta\sin(2n\theta)\, d\theta\\
&=0.\nonumber
\end{align*}
Also, a similar procedure gives
\begin{align}\label{am-1}
a_m&=2\int_{0}^{1}\frac{\cos(2\pi mx)}{2n}\, dx+\frac{4(-1)^n}{\pi}\int_{0}^{\pi/2}\theta\sin(2n\theta)\cos\left(4\pi m\cos\theta\right)\, d\theta\\
&=\frac{4(-1)^n}{\pi}\int_{0}^{\pi/2}\theta\sin(2n\theta)\bigg(J_{0}(4\pi m)+2\sum_{j=1}^{\infty}(-1)^{j}J_{2j}(4\pi m)\cos(2j\theta)\bigg)\, d\theta,\nonumber
\end{align}
where in the last step, formula 10.12.3 in \cite[p.~226]{olver-2010a} was used. Now
\begin{align}\label{am-2}
\frac{4(-1)^nJ_{0}(4\pi m)}{\pi}\int_{0}^{\pi/2}\theta\sin(2n\theta)d\theta=-\frac{J_{0}(4\pi m)}{n},
\end{align}
and
{\allowdisplaybreaks\begin{align}\label{calculation-p}
&\frac{8(-1)^n}{\pi}\int_{0}^{\pi/2}\theta\sin(2n\theta)\sum_{j=1}^{\infty}(-1)^{j}J_{2j}(4\pi m)\cos(2j\theta)\, d\theta\\
&=\frac{8(-1)^n}{\pi}\bigg(\sum_{j=1\atop {j\neq n}}^{\infty}(-1)^{j}J_{2j}(4\pi m)\int_{0}^{\pi/2}\theta\sin(2n\theta)\cos(2j\theta)\, d\theta\nonumber\\
&\quad+(-1)^nJ_{2n}(4\pi m)\int_{0}^{\pi/2}\theta\sin(2n\theta)\cos(2n\theta)\, d\theta\bigg)\nonumber\\
%&=\frac{8(-1)^n}{\pi}\bigg(\frac{(-1)^nn\pi}{4}\sum_{j=1\atop {j\neq n}}^{\infty}\frac{J_{2j}(4\pi m)}{j^2-n^2}+\frac{(-1)^{n+1}\pi J_{2n}(4\pi m)}{16n}\bigg)\nonumber
&=2n\sum_{j=1\atop {j\neq n}}^{\infty}\frac{J_{2j}(4\pi m)}{j^2-n^2}-\frac{J_{2n}(4\pi m)}{2n}\nonumber\\
&=-\sum_{j=1\atop {j\neq n}}^{\infty}J_{2j}(4\pi m)\left(\frac{1}{n+j}+\frac{1}{n-j}\right)-\frac{J_{2n}(4\pi m)}{2n}\nonumber\\
&=-\bigg(\sum_{k=n+1\atop {k\neq 2n}}^{\infty}\frac{J_{2k-2n}(4\pi m)}{k}+\sum_{k=-\infty\atop {k\neq 0}}^{n-1}\frac{J_{2n-2k}(4\pi m)}{k}\bigg)-\frac{J_{2n}(4\pi m)}{2n}\nonumber\\
&=-\bigg(\sum_{k=n+1\atop {k\neq 2n}}^{\infty}\frac{J_{2k-2n}(4\pi m)}{k}-\sum_{k=1}^{\infty}\frac{J_{2n+2k}(4\pi m)}{k}+\sum_{k=1}^{n-1}\frac{J_{2n-2k}(4\pi m)}{k}\bigg)
-\frac{J_{2n}(4\pi m)}{2n}\nonumber\\
&=\sum_{k=1}^{\infty}\frac{J_{2n+2k}(4\pi m)}{k}-\sum_{k=1\atop {k\neq n}}^{\infty}\frac{J_{2n-2k}(4\pi m)}{k}.\nonumber
\end{align}}
Together with \eqref{am-1} and \eqref{am-2}, this implies
\begin{align*}
a_m=\sum_{k=1}^{\infty}\frac{J_{2n+2k}(4\pi m)-J_{2n-2k}(4\pi m)}{k},
\end{align*}
which, according to \eqref{form-PQ1}, is equal to $P_{2n}(4\pi m)$. 

Lastly since $f(x, n)$ is $1$-periodic and is of class $C^{1}$, its Fourier series converges to it. This proves the lemma.
\end{proof}

The result on $\left.\frac{\partial}{\partial \nu}J_{\nu}(4 \pi m )\right|_{\nu=2n}$, arising as 
Fourier coefficients of a function involving Chebyshev polynomials, inverse trigonometric and some infinite series, is now obtained.

\begin{lemma}\label{lemma-two5}
Let $g(y, r, x)$ be defined in \eqref{gynx}. For $n \in \mathbb{N}$ and $0 < x < 1$, the identity 
{\allowdisplaybreaks\begin{align*}
  &\sum_{m=1}^{\infty} \left.\frac{\partial }{\partial \nu}J_{\nu}(4 \pi m)\right|_{\nu=2n} \cos(2 \pi m x)\\
  & =\frac{(-1)^{n}}{4 \pi} \bigg\{ \sin^{-1} \left( \frac{x}{2} \right) U_{2n-1} \left( \frac{x}{2} \right) + 
   \sin^{-1} \left( \frac{x+1}{2} \right) U_{2n-1} \left( \frac{x+1}{2} \right)  \\
   & \quad \quad \quad \quad\quad + \sin^{-1} \left( \frac{1-x}{2} \right) U_{2n-1} \left( \frac{1-x}{2} \right) + 
    \sin^{-1} \left( \frac{2-x}{2} \right) U_{2n-1} \left( \frac{2-x}{2}  \right) \bigg\} \\
    &\quad+\frac{(-1)^{n+1}}{4^{n+1}} \bigg(\sum_{m=1}^{\infty}g(m, n, x)+\sum_{m=1}^{\infty}g(m, n, 1-x)\bigg)
    \end{align*}}
		holds.
\end{lemma}
\begin{proof}
Let 
\begin{align*}
h_{1}(x, n)&=\frac{(-1)^{n}}{4 \pi} \bigg\{ \sin^{-1} \left( \frac{x}{2} \right) U_{2n-1} \left( \frac{x}{2} \right) + 
   \sin^{-1} \left( \frac{x+1}{2} \right) U_{2n-1} \left( \frac{x+1}{2} \right)\\
	& \quad \quad  + \sin^{-1} \left( \frac{1-x}{2} \right) U_{2n-1} \left( \frac{1-x}{2} \right) + 
    \sin^{-1} \left( \frac{2-x}{2} \right) U_{2n-1} \left( \frac{2-x}{2}  \right) \bigg\},\\
h_{2}(x, n)&=\frac{(-1)^{n+1}}{4^{n+1}} \bigg(\sum_{m=1}^{\infty}g(m, n, x)+\sum_{m=1}^{\infty}g(m, n, 1-x)\bigg),
\end{align*}
and let 
\begin{equation*}
h(x, n)=h_{1}(x, n)+h_{2}(x, n).
\end{equation*}
Note that $h(1-x, n)=h(x, n)$. 
The periodization of $h(x, n)$, as a function of $x$, based on its values in $[0,1)$ makes it 
an even function of $x$. Hence its Fourier series is given by
\begin{equation*}
b_0+\sum_{\ell=1}^{\infty}b_{\ell}\cos(2\pi \ell x),
\end{equation*}
where
\begin{align*}
b_0=\int_{0}^{1}h(x, n)\, dx, \hspace{6mm} b_{\ell}=2\int_{0}^{1}h(x, n)\cos(2\pi \ell x)\, dx,
\end{align*}
for $\ell\geq 1$. The coefficient $b_{0}$ is computed first. Observe that the change of variable $x=2\cos\theta$ yields 
\begin{equation*}
\int_{0}^{1}U_{2n-1}\left(\frac{x}{2}\right)\sin^{-1}\left(\frac{x}{2}\right)\, dx=2\int_{\pi/3}^{\pi/2}\left(\frac{\pi}{2}-\theta\right)\sin(2n\theta)\, d\theta.
\end{equation*}
Computing the other three integrals in the similar way, it is seen that
\begin{align}\label{h1}
\int_{0}^{1}h_1(x, n)\, dx=\frac{(-1)^n}{\pi}\int_{0}^{\pi/2}\left(\frac{\pi}{2}-\theta\right)\sin(2n\theta)\, d\theta=\frac{(-1)^{n}}{4n}.
\end{align}
Note that
\begin{align*}
\int_{0}^{1}h_2(x, n)\, dx=\frac{2(-1)^{n+1}}{4^{n+1}}\int_{0}^{1}\sum_{m=1}^{\infty}g(m, n, x)\, dx=\frac{(-1)^{n+1}}{2^{2n+1}}\sum_{m=1}^{\infty}\int_{0}^{1}g(m, n, x)\, dx,
\end{align*}
where the interchange of the order of summation and integration is valid because of absolute convergence. 

Let $y=m+x-1$ in the above integral and observe that $(m-1+x)(m+3+x)=(m+1+x)^2-4$ 
so as to have
\begin{align}\label{h2}
\int_{0}^{1}h_2(x, n)\, dx&=\frac{(-1)^{n+1}}{2^{2n+1}}\sum_{m=1}^{\infty}\int_{m-1}^{m}\frac{\big(y+2-\sqrt{(y+2)^2-4}\big)^{2n}}{\sqrt{(y+2)^2-4}}\, dy\\
&=\frac{(-1)^{n+1}}{2^{2n+1}}\int_{0}^{\infty}\frac{\big(y+2-\sqrt{(y+2)^2-4}\big)^{2n}}{\sqrt{(y+2)^2-4}}\, dy.\nonumber
\end{align}
The change of variable $y+2=\sec\theta$ transforms the above integral into
\begin{align*}
\int_{0}^{\infty}\frac{\big(y+2-\sqrt{(y+2)^2-4}\big)^{2n}}{\sqrt{(y+2)^2-4}}\, dy&=2^{2n}\int_{0}^{\pi/2}(\sec\theta-\tan\theta)^{2n}\sec\theta\, d\theta\\
&=2^{2n}\int_{0}^{1}t^{2n-1}\, dt\nonumber\\
&=\frac{2^{2n-1}}{n},
\end{align*}
by another substitution $\sec\theta-\tan\theta=t$. Along with \eqref{h1} and \eqref{h2}, this yields
\begin{equation}\label{b0}
b_0=0.
\end{equation}
Next, $b_{\ell}$ is computed. Using exactly the same approach as in \eqref{am-1}, it is seen that
{\allowdisplaybreaks\begin{align*}
2\int_{0}^{1}h_1(x, n)\cos(2\pi \ell x)\, dx&=\frac{2(-1)^n}{\pi}\int_{0}^{\pi/2}\left(\frac{\pi}{2}-\theta\right)\sin(2n\theta)\cos\left(4\pi \ell\cos\theta\right)\, d\theta\\
&=-\frac{2}{\pi}\int_{0}^{\pi/2}\theta\sin(2n\theta)\cos\left(4\pi \ell\sin\theta\right)\, d\theta.\nonumber
\end{align*}}
Using the formula \cite[p.~226]{olver-2010a} 
\begin{equation*}
\cos(z\sin\theta)=J_{0}(z)+2\sum_{k=1}^{\infty}J_{2k}(z)\cos(2k\theta),
\end{equation*}
valid for $z, \theta\in\mathbb{C}$, and performing a calculation similar to that in \eqref{calculation-p}, one sees that
\begin{align}\label{h1-eval}
2\int_{0}^{1}h_1(x, n)\cos(2\pi \ell x)\, dx=-\frac{1}{2}\sum_{k=1}^{\infty}\frac{(-1)^k}{k}\left(J_{2n+2k}(4\pi\ell)-J_{2n-2k}(4\pi\ell)\right).
\end{align}
The difficult task now is to evaluate the corresponding integral involving $h_2(x, n)$. This 
integral is first written in a convenient form using an approach similar to that on 
the previous page, namely,
\begin{align}\label{h2-eval}
&2\int_{0}^{1}h_2(x, n)\cos(2\pi \ell x)\, dx\\&=\frac{(-1)^{n+1}}{2^{2n}}\int_{0}^{1}\sum_{m=1}^{\infty}g(m, n, x)\cos(2\pi\ell x)\, dx\nonumber\\
&=\frac{(-1)^{n+1}}{2^{2n}}\int_{0}^{\infty}\frac{\big(y+2-\sqrt{(y+2)^2-4}\big)^{2n}\cos(2\pi\ell y)}{\sqrt{(y+2)^2-4}}\, dy,\nonumber\\
&= (-1)^{n+1}\int_{0}^{\pi/2}(\sec\theta-\tan\theta)^{2n}\sec\theta\cos(4\pi\ell\sec(\theta))\, d\theta\nonumber\\
&= (-1)^{n+1}\int_{0}^{\infty}e^{-2n\phi}\cos\left(4\pi\ell\cosh\phi\right)\, d\phi,\nonumber
\end{align}
where the substitution $\sec\theta-\tan\theta=e^{-\phi}$ was used in the last step. 

C. V.~Coates \cite[p.~260]{coates} showed that the integral
\begin{equation}\label{coaint}
(-1)^{n+1}\int_{0}^{\infty}e^{-2n\phi}\cos\left(u\cosh\phi\right)\, d\phi
\end{equation}
satisfies the non-homogeneous second-order linear differential equation
\begin{equation}\label{de}
\frac{d^{2}y}{du^2}+\frac{1}{u}\frac{dy}{du}+y\left(1-\frac{4n^2}{u^2}\right)=\frac{2n(-1)^n\cos(u)}{u^2}.
\end{equation}
An integral whose real part is equal to above integral also appears in Whipple's work \cite[p.~106]{whipple-1917a} on diffraction theory. 
Yet no one seems to have explicitly evaluated it. Asatryan \cite[Equation (A 9)]{asatryan} has obtained an evaluation of this integral in terms of a nested finite sum involving factorials. The following new evaluation of \eqref{coaint} proved below for $u>0$ and $n\in\mathbb{N}$, namely
\begin{align}\label{integral-id}
&(-1)^{n+1}\int_{0}^{\infty}e^{-2n\phi}\cos\left(u\cosh\phi\right)\, d\phi\\
&=\left(\log\left(\tfrac{u}{2}\right)-\psi(2n+1)\right)J_{2n}(u)\nonumber\\
&\quad-\frac{1}{2}\sum_{k=1}^{\infty}\frac{(-1)^k}{k}\left(J_{2n+2k}(u)+J_{2n-2k}(u)\right)-\sum_{k=1}^{\infty}\frac{(-1)^k}{k+2n}J_{2n+2k}(u),\nonumber
\end{align}
will be useful in completing the proof of this lemma.

To establish this identity, the existence and uniqueness theorem for second-order linear differential equations is employed. Let 
$w(u, n)$ denote the right-hand side of \eqref{integral-id}. By brute force it can be verified (although it is fairly tedious) that  
\begin{align}\label{tedious}
&\frac{d^{2}w}{du^2}+\frac{1}{u}\frac{dw}{du}+w\left(1-\frac{4n^2}{u^2}\right)\\
&=\frac{1}{u}\left(J_{2n-1}(u)-J_{2n+1}(u)\right)-\frac{4}{u^2}\sum_{k=1}^{\infty}(-1)^{k}kJ_{2n+2k}(u)\nonumber\\
&\quad-\frac{2}{u^2}\bigg(\sum_{k=1}^{\infty}(-1)^{k}(k+2n)J_{2n+2k}(u)+\sum_{k=1}^{\infty}(-1)^{k}(k-2n)J_{2n-2k}(u)\bigg).\nonumber
\end{align}
 The standard formulas \cite[p.~222, 10.6.1]{olver-2010a}
\begin{equation*}
\frac{d}{dz}J_{\nu}(z)=\frac{1}{2}\left(J_{\nu-1}(z)-J_{\nu+1}(z)\right)
\end{equation*}
and
\begin{equation}\label{recur}
J_{\nu-1}(z)+J_{\nu+1}(z)=\frac{2\nu}{z}J_{\nu}(z),
\end{equation}
 are used in this verification. Substituting $j=k-2n$ in the sum on the extreme right in 
\eqref{tedious} and simplifying, the right-hand side of \eqref{tedious} is seen to be equal to
{\allowdisplaybreaks\begin{align}\label{ha-begin}
&\frac{1}{u}\left(J_{2n-1}(u)-J_{2n+1}(u)\right)\\
&\quad-\frac{2}{u^2}\bigg(\sum_{k=1}^{\infty}(-1)^{k}(2k+2n)J_{2n+2k}(u)-2n\sum_{k=1}^{\infty}(-1)^kJ_{2n+2k}(u)\bigg)\nonumber\\
&\quad-\frac{2}{u^2}\bigg(\sum_{k=1}^{\infty}(-1)^{k}(k+2n)J_{2n+2k}(u)+\sum_{j=1-2n}^{\infty}(-1)^jjJ_{-2j-2n}(u)\bigg)\nonumber\\
&=\frac{1}{u}\left(J_{2n-1}(u)-J_{2n+1}(u)\right)-\frac{4}{u^2}\sum_{k=1}^{\infty}(-1)^{k}(2k+2n)J_{2n+2k}(u)\nonumber\\
&\quad+\frac{4n}{u^2}\sum_{k=1}^{\infty}(-1)^kJ_{2n+2k}(u)+\frac{2}{u^2}\sum_{k=1}^{2n-1}(-1)^kkJ_{2n-2k}(u).\nonumber
\end{align}}
Now note that from \cite[p.~270]{neilsen-1904a},
\begin{equation}\label{ha-begin1}
\sum_{k=1}^{\infty}(-1)^{k}(2k+2n)J_{2n+2k}(u)=\frac{u}{2}J_{2n-1}(u)-2nJ_{2n}(u),
\end{equation}
whereas
\begin{equation}\label{ha-begin2}
\sum_{k=1}^{\infty}(-1)^kJ_{2n+2k}(u)=M_{2n}(u, u)-J_{2n}(u),
\end{equation}
where $M_{\nu}(s, z)$ is the Lommel function of two variables defined by \cite[p.~537, 16.5(5)]{watson-1966a} \footnote{The conventional notation $U_{n}(s, z)$ is avoided 
so as to not get confused with the one for the Chebyshev polynomial $U_n(x)$.}
\begin{equation*}
M_{\nu}(s,z)=\sum_{m=0}^{\infty}(-1)^m\left(\frac{s}{z}\right)^{\nu+2m}J_{\nu+2m}(z).
\end{equation*}
Further,
\begin{align}\label{ha1}
\sum_{k=1}^{2n-1}(-1)^kkJ_{2n-2k}(u)&=-\frac{1}{2}\sum_{k=1}^{2n-1}(-1)^{k}(2n-2k)J_{2n-2k}(u)\\
&\quad+n\bigg(\sum_{k=1}^{n}(-1)^{k}J_{2n-2k}(u)+\sum_{k=n+1}^{2n-1}(-1)^{k}J_{2n-2k}(u)\bigg)\nonumber
\end{align}
From \cite[p.~383, (57.4.15)]{hansen-1975a},
\begin{align}\label{ha2}
\sum_{k=1}^{2n-1}(-1)^{k}(2n-2k)J_{2n-2k}(u)=\frac{u}{2}\left(J_{2n+1}(u)-J_{-2n+1}(u)\right)-2nJ_{2n}(u).
\end{align}
Also from \cite[p.~379, (57.1.25)]{hansen-1975a},
\begin{align}\label{ha3}
\sum_{k=1}^{n}(-1)^{k}J_{2n-2k}(u)=\frac{(-1)^{n}}{2}(\cos u+J_{0}(u))-M_{2n}(u, u),
\end{align}
whereas using formula (57.1.19) from \cite[p.~378]{hansen-1975a},
\begin{align}\label{ha4}
\sum_{k=n+1}^{2n-1}(-1)^{k}J_{2n-2k}(u)&=(-1)^n\sum_{j=1}^{n-1}(-1)^{j}J_{2j}(u)\\
&=(-1)^{n+1}J_{0}(u)+\frac{(-1)^n}{2}(\cos u+J_{0}(u))-M_{2n}(u, u).\nonumber
\end{align}
Substituting \eqref{ha2}, \eqref{ha3} and \eqref{ha4} in \eqref{ha1} yields
\begin{align}\label{ha-final}
\sum_{k=1}^{2n-1}(-1)^kkJ_{2n-2k}(u)&=n(-1)^n\cos u+nJ_{2n}(u)\\
&\quad-\frac{u}{4}(J_{2n+1}(u)+J_{2n-1}(u))-2nM_{2n}(u, u).\nonumber
\end{align}
Substituting \eqref{ha-begin1}, \eqref{ha-begin2} and \eqref{ha-final} in \eqref{ha-begin}, and using \eqref{tedious} leads to
{\allowdisplaybreaks\begin{align*}
&\frac{d^{2}w}{du^2}+\frac{1}{u}\frac{dw}{du}+w\left(1-\frac{4n^2}{u^2}\right)\\
&=\frac{2n(-1)^n\cos u}{u^2}+\frac{6n}{u^2}J_{2n}(u)-\frac{3}{2u}(J_{2n-1}(u)+J_{2n+1}(u))\nonumber\\
&=\frac{2n(-1)^n\cos u}{u^2},\nonumber
\end{align*}}
where the last step follows from \eqref{recur}.

Thus the right-hand side of \eqref{integral-id} also satisfies the differential equation \eqref{de}. It is easy to see that 
both sides of \eqref{integral-id} agree with each other at $u=0$, and their derivatives agree at $u=0$ as well. 
The differential equation has a regular singular point at $x=0$ with indices $\pm n$. The standard structure of the 
space of solutions shows that if two  solutions $y_{1}, \, y_{2}$ have matching values and derivatives  at $0$, then 
$y_{1} \equiv y_{2}$. This implies that \eqref{integral-id} is true for all $u>0$ and $n\in\mathbb{N}$.

Now substitute \eqref{integral-id} with $u=4\pi\ell$ in \eqref{h2-eval}, and then combine the resultant with 
\eqref{h1-eval} thereby obtaining
\begin{align}\label{bm}
b_{\ell} &= \left( \log \left( 2\pi\ell \right) - \psi(2n+1) \right)  J_{2n}(4\pi\ell) - 
 \sum_{k=1}^{\infty} \frac{(-1)^{k} (2k+ 2n)}{k(k+ 2n)} J_{2k + 2n}(4\pi\ell)\\
&= \left.\frac{\partial}{\partial \nu} J_{\nu}(4\pi\ell)\right|_{\nu=2n},\nonumber
\end{align}
the last step resulting from \eqref{besselj-der}. Since $h(x, n)$ is $1$-periodic and is of class $C^{1}$, 
its Fourier series converges to it. This fact, along with the expressions for $b_0$ and 
$b_{\ell}$ in \eqref{b0} and \eqref{bm} complete the proof of Lemma \ref{lemma-two5}.
\end{proof}
Having proved the above lemmas, we are now ready to complete the proof of Theorem \ref{thm-main-0}. Substitute the results 
of Lemmas \ref{lemma-two4} and \ref{lemma-two5} in Lemma \ref{lemma-bessel-00}, and use the elementary identity 
$\sin^{-1}(x)+\cos^{-1}(x)=\frac{\pi}{2}$ to obtain
\begin{align}\label{finalaxn}
A(n,x)&=(-1)^n\pi\sum_{m=1}^{\infty} Y_{2n}(4\pi m)\cos(2\pi mx)\\
&\quad+\frac{(-1)^{n+1}}{2n}+\frac{1}{2^{2n+1}}\bigg(\sum_{m=1}^{\infty}g(m, n, x)+\sum_{m=1}^{\infty}g(m, n, 1-x)\bigg)\nonumber\\
&\quad-\frac{1}{4}\left(U_{2n-1}\left(\frac{x}{2}\right)+U_{2n-1}\left(\frac{x+1}{2}\right)+U_{2n-1}\left(\frac{1-x}{2}\right)+U_{2n-1}\left(\frac{2-x}{2}\right)\right).\nonumber
\end{align}
Theorem \ref{thm-main-0} now follows from substituting the above representation of $A(n, x)$ in \eqref{formula-ber11} and then 
making use of the fact that $U_{2n-1}(-y)=-U_{2n-1}(y)$.

 \section{Recovering Zagier's formula}
 \label{sec-recovery}
 
 This section shows that Zagier's exact formula \eqref{zagier-sum} follows 
by letting $x \to 1$ on both sides of the result of Theorem \ref{thm-main-0}. As is clear from \eqref{Y-diverge}, 
the series $\begin{displaystyle} (-1)^n\pi\sum_{m=1}^{\infty} Y_{2n}(4 \pi m) \end{displaystyle}$ 
 diverges, and so the passage to the limit requires some care.  
 
 %For large values of $|z|$ and $| \arg z | < \pi$, the asymptotic behavior of $Y_{\nu}(z)$ is given by
 %\cite[p. 199, (2)]{watson-1966a}
 %\begin{eqnarray*}
 %Y_{\nu}(z) & \sim & \left( \frac{2}{\pi z} \right)^{1/2} \left[ \sin \left( z - \frac{\pi \nu}{2} - \frac{\pi}{4} \right) 
 %\sum_{k=0}^{\infty} \frac{\Gamma(\nu + 2 k + \tfrac{1}{2})}{(2k)! \, \Gamma(\nu - 2k + \tfrac{1}{2} )} 
 %\frac{(-1)^{k}}{(2 z)^{2k}} \right. \\
 %&  & \left. \quad \qquad \qquad  + \cos \left( z - \frac{\pi \nu}{2} - \frac{\pi}{4} \right) 
 % \sum_{k=0}^{\infty} \frac{\Gamma(\nu + 2 k + \tfrac{3}{2})}{(2k+1)! \, \Gamma(\nu - 2k - \tfrac{1}{2} )} 
 %\frac{(-1)^{k}}{(2 z)^{2k+1}} \right].
 %\end{eqnarray*}
 %\noindent
 From \eqref{asymbess1}, we easily get the big-O bound 
 \begin{equation*}
 \left( (-1)^{n} \pi Y_{2n}( 4 \pi m ) + \frac{1}{2 \sqrt{m}} \right) \cos( 2 \pi m x) = O_{n}\left( m^{-3/2} \right).
 \end{equation*}
 \noindent
This suggests replacing the term $(-1)^{n} \pi Y_{2n}(4 \pi m )$ 
 by $(-1)^{n} \pi Y_{2n}(4 \pi m ) + 1/(2 \sqrt{m})$ in the statement of Theorem \ref{thm-main-0}. The uniform 
convergence of the series 
\begin{equation*}
\sum_{m=1}^{\infty}\left((-1)^{n} \pi Y_{2n}(4 \pi m ) + \frac{1}{2 \sqrt{m}}\right)\cos(2\pi mx)
\end{equation*}
for $x\in(0,1)$ will then allow us to move the limit inside the sum. The consequences of adding the term $1/(2 \sqrt{m})$ 
are examined first. Note that
 \begin{align}\label{modification}
 &(-1)^{n} \pi\sum_{m=1}^{\infty} Y_{2n}(4 \pi m ) \cos( 2 \pi m x) \\
& = \sum_{m=1}^{\infty} \left( (-1)^{n} \pi Y_{2n}(4 \pi m ) + \frac{1}{2 \sqrt{m}} \right)  \cos( 2 \pi m x) - \sum_{m=1}^{\infty} \frac{\cos( 2 \pi m x)}{2 \sqrt{m}}. \nonumber
   \end{align}
   
  An appropriate modification of the second  series on the right  is obtained from the Hurwitz zeta function 
   %\begin{equation}
   %\zeta(s,x) = \sum_{n = 0}^{\infty} \frac{1}{(n+x)^{s}}
   %\end{equation}
   %\noindent
   and its representation  \cite[p. 257, Theorem 12.6]{apostol-1998a} 
   \begin{equation*}
   \zeta(1-s,x) = \frac{2 \Gamma(s)}{(2 \pi)^{s}} 
   \left( \cos \left( \frac{\pi s }{2 } \right) \sum_{m=1}^{\infty} \frac{\cos( 2 \pi m x)}{m^{s}} + 
    \sin \left( \frac{\pi s }{2 } \right) \sum_{m=1}^{\infty} \frac{\sin( 2 \pi m x)}{m^{s}} \right),
    \end{equation*}
    \noindent
    valid for $0 < x \leq 1$ and Re$(s) > 1$. This expansion is also valid for Re$(s) > 0$ provided 
    $x \neq 1$.  The special case $s  = \tfrac{1}{2}$ is used to obtain 
    \begin{equation*}
    \zeta \left( \tfrac{1}{2}, x \right) = \sum_{m=1}^{\infty} \frac{\cos(2 \pi m x)}{\sqrt{m}} + 
     \sum_{m=1}^{\infty} \frac{\sin(2 \pi m x)}{\sqrt{m}}.
     \end{equation*}
     Thus
      \begin{eqnarray*}
 (-1)^{n} \pi \sum_{m=1}^{\infty} Y_{2n}(4 \pi m ) \cos( 2 \pi m x) & = & 
  \sum_{m=1}^{\infty} \left( (-1)^{n} \pi Y_{2n}(4 \pi m ) + \frac{1}{2 \sqrt{m}} \right)  \cos( 2 \pi m x)   \\
   & & - \frac{1}{2} \zeta \left( \tfrac{1}{2}, x \right) + \frac{1}{2}  \sum_{m=1}^{\infty} \frac{\sin(2 \pi m x)}{ \sqrt{m}}. 
   \end{eqnarray*}
   The last series is simplified using the identity \cite[Equation (9)]{popov-1936a}
  \begin{equation*}
  \sum_{m=1}^{\infty} J_{\nu}(2 \pi m x) = \frac{1}{2 \pi x} - \frac{1}{\pi} x^{\nu} \sin \left( \frac{\pi \nu}{2} \right) 
  \sum_{m=1}^{\infty} \frac{1}{( m + \sqrt{m^{2} - x^{2}})^{\nu} \, \sqrt{m^{2} - x^{2}}}, 
  \end{equation*}
  \noindent
  valid for Re$(\nu) > 0$ and $0 < x < 1$, and which may be established by applying the Poisson summation formula \eqref{poisson} 
  to the function $J_{\nu}(2 \pi x |t|)$.  Now use the particular case $\nu = \tfrac{1}{2}$ and the fact \cite[p.~924, entry 8.464.1]{gradshteyn-2015a} that
  \begin{equation*}
  J_{1/2}(z) = \sqrt{ \frac{2}{\pi z}} \sin z 
  \end{equation*}
  \noindent
  to obtain 
  \begin{equation}
  \label{series-007}
  \sum_{m=1}^{\infty} \frac{\sin( 2 \pi m x)}{\sqrt{m}} = \frac{1}{2 \sqrt{x}} - 
  \frac{x}{\sqrt{2}} \sum_{m=1}^{\infty} \frac{1}{\sqrt{m + \sqrt{m^{2} - x^{2}}} \,\, \sqrt{m^{2}-x^{2}}}.
  \end{equation}
  \noindent
  It follows that 
  \begin{eqnarray*}
 (-1)^{n} \pi\sum_{m=1}^{\infty} Y_{2n}(4 \pi m ) \cos( 2 \pi m x) & = & 
  \sum_{m=1}^{\infty} \left( (-1)^{n} \pi Y_{2n}(4 \pi m )   + \frac{1}{2 \sqrt{m}} \right)  \cos( 2 \pi m x)     \\
 & &    -  \frac{1}{2} \zeta \left( \tfrac{1}{2}, x \right)  
     +   \frac{1}{4 \sqrt{x}}  \nonumber \\
     & & - 
    \frac{x}{2 \sqrt{2}} \sum_{m=1}^{\infty} \frac{1}{\sqrt{m + \sqrt{m^{2} - x^{2}}} \,\, \sqrt{m^{2}-x^{2}}}.
   \end{eqnarray*}
   The result of Theorem \ref{thm-main-0} is now expressed as 
    {\allowdisplaybreaks\begin{eqnarray*}
   %\label{mess-new} && \\
   B_{2n}^{*}(x) & = & \sum_{m=1}^{\infty} \left( (-1)^{n} \pi Y_{2n}(4 \pi m ) + \frac{1}{2 \sqrt{m}} \right) \cos(2 \pi m x)  -  \frac{1}{2} \zeta \left( \tfrac{1}{2}, x \right)  
     +   \frac{1}{4 \sqrt{x}}  
   \nonumber  \\
     %& &   \nonumber \\
     & & - \left( \frac{x}{2 \sqrt{2} \sqrt{1 + \sqrt{1-x^{2}}} \, \sqrt{1-x^{2}}} - 
     \frac{1}{2^{2n+1}} \frac{( 3 - x - \sqrt{(1-x)(5-x)})^{2n}}{\sqrt{(1-x)(5-x)}} \right)  \nonumber \\
     & & - \frac{x}{2 \sqrt{2}} \sum_{m=1}^{\infty} \frac{1}{ \sqrt{ m+1 + 
     \sqrt{ (m+1)^{2} - x^{2}}} \, \sqrt{(m+1)^{2} - x^{2}}}  \nonumber \\
     & & + \frac{1}{4} \left( U_{2n-1} \left( \frac{x+1}{2} \right) + U_{2n-1} \left( \frac{x}{2} \right) + 
     U_{2n-1} \left( \frac{x-1}{2} \right) + U_{2n-1} \left( \frac{x-2}{2} \right) \right)  \nonumber \\
     & & + \frac{1}{2^{2n+1}} \left(\sum_{m=1}^{\infty} g(m,n,x)+\sum_{m=1}^{\infty} g(m+1,n,1-x)\right).
     \end{eqnarray*}}
     Now let $x \uparrow  1$, use the facts that $U_{2n-1}(0) = 0, \, U_{2n-1}(1) = 2n$ and that $U_{2n-1}(x)$ 
		is an odd function of $x$, along with the value 
     \begin{align*}
     &\lim_{x \to 1^{-}} \left( 
     \frac{x}{2 \sqrt{2} \sqrt{1 + \sqrt{1-x^{2}}} \, \sqrt{1-x^{2}}} - 
     \frac{1}{2^{2n+1}} \frac{( 3 - x - \sqrt{(1-x)(5-x)})^{2n}}{\sqrt{(1-x)(5-x)}} \right) \\
     &=\frac{n}{2} - \frac{1}{4 \sqrt{2}},
     \end{align*}
     \noindent
     to obtain 
     \begin{eqnarray*}
     B_{2n}^{*}(1) & = & \sum_{m=1}^{\infty} \left( (-1)^{n} \pi Y_{2n}(4 \pi  m) + \frac{1}{2 \sqrt{m}} \right) 
     - \frac{1}{2} \zeta \left( \frac{1}{2}\right) + \frac{1}{4}  \label{form-mess-11} \\
     & & + \frac{1}{4 \sqrt{2}} - \frac{1}{2 \sqrt{2}} \sum_{m=1}^{\infty} 
     \frac{1}{\sqrt{m (m+2)} \, \sqrt{m+1 + \sqrt{m(m+2)}}}  \nonumber \\
     & & + \sum_{m=1}^{\infty} \frac{1}{\sqrt{m(m+4)}} \left( \frac{ \sqrt{m+4} - \sqrt{m}}{2} \right)^{4n},
     \nonumber 
     \end{eqnarray*}
     \noindent
     where the identity $m+2 - \sqrt{m(m+4)} = \tfrac{1}{2} ( \sqrt{m+4} - \sqrt{m} )^{2}$ was used in the 
     simplification of the resulting series on the extreme right side.   The last step is to use 
     \begin{equation*}
     m+1 - \sqrt{m(m+2)} = \frac{1}{2} \left( \sqrt{m+2} - \sqrt{m} \right)^{2}
     \end{equation*}
     \noindent
     to evaluate, as a telescoping series, 
     \begin{equation*}
     \sum_{m=1}^{\infty} 
     \frac{1}{\sqrt{m (m+2)} \, \sqrt{m+1 + \sqrt{m(m+2)}}}  = \frac{\sqrt{2}+1}{2}.
     \end{equation*}
     \noindent 
     Finally use the result $B_{2n}^{*}(1) = B_{2n}^{*} + n $ \cite[Formula (10.16)]{dixit-2014a}, \cite[p.~5]{zagier-1998a} to
     obtain Zagier's formula \eqref{zagier-sum}.

 \section{Periodicity and the odd-index case}
 \label{sec-periodicity}

 The original motivation that led us to the study of the modified Bernoulli numbers was the curious phenomenon 
 that $\{ B_{2n+1}^{*}: \, n \in \mathbb{N} \}$ is a periodic sequence. Since the proof of Theorem 
 \ref{thm-main-1} is similar to that of Theorem \ref{thm-main-0}, it is omitted. However, a new derivation 
of the periodicity of $\{ B_{2n+1}^{*}\}$ using Theorem \ref{thm-main-1} is given below.

\begin{corollary}
 For $n \in \mathbb{N}$,
 \begin{equation*}
 B_{2n+1}^{*} =  \frac{(-1)^{n}}{4} + \frac{1}{2} U_{2n} \left( \frac{1}{2} \right) = 
 \frac{(-1)^{n}}{4}  + \frac{1}{\sqrt{3}} \sin \left( \frac{(2n+1) \pi}{3} \right).
 \end{equation*}
 \end{corollary}
 \begin{proof}
 The series in Theorem \ref{thm-main-1} may be written as 
 \begin{align*}
 &(-1)^{n} \pi\sum_{m=1}^{\infty} Y_{2n+1}(4 \pi m ) \sin(2 \pi m  x)\\
 & = \sum_{m=1}^{\infty} \left( (-1)^{n} \pi Y_{2n+1}(4 \pi m ) + \frac{1}{2 \sqrt{m}} \right) \sin( 2 \pi m  x) 
 - \frac{1}{2} \sum_{m=1}^{\infty} \frac{\sin( 2 \pi m x )}{\sqrt{m}}  \\
 &= \sum_{m=1}^{\infty} \left( (-1)^{n} \pi Y_{2n+1}(4 \pi m ) + \frac{1}{2 \sqrt{m}} \right)\sin(2 \pi m  x)  - \frac{1}{4 \sqrt{x}} \\
 &\quad+ \frac{x}{2 \sqrt{2}} \sum_{m=1}^{\infty} \frac{1}{\sqrt{m + \sqrt{m^{2} - x^{2}}} \, \sqrt{m^{2}-x^{2}}},
 \end{align*}
 \noindent
 using \eqref{series-007}. Now replace the above representation in the formula from Theorem \ref{thm-main-1} to obtain
\begin{align*}
 B_{2n+1}^{*}(x)  &= \sum_{m=1}^{\infty} \left( (-1)^{n} \pi Y_{2n+1}(4 \pi m ) + \frac{1}{2 \sqrt{m}} \right) 
 \sin(2 \pi m x) \\
 &\quad+\frac{x}{2 \sqrt{2}} \sum_{m=1}^{\infty} \frac{1}{\sqrt{m + \sqrt{m^{2} - x^{2}}} \, \sqrt{m^{2}-x^{2}}}\\
&\quad  +  \frac{1}{4} \left( \frac{(2 + x - \sqrt{x(4+x)})^{2n+1}}{2^{2n} \sqrt{x(4+x)}} - \frac{1}{\sqrt{x}} \right)\\
&\quad+\frac{1}{4} \left( U_{2n} \left( \frac{x+1}{2} \right) +  U_{2n} \left( \frac{x}{2} \right)+ U_{2n} \left( \frac{x-1}{2} \right) +  U_{2n} \left( \frac{x-2}{2} \right)  \right) \\
 &\quad+ \frac{1}{2^{2n+2}}\left(\sum_{m=1}^{\infty} g\left(m+1, n+\tfrac{1}{2},x\right)-\sum_{m=1}^{\infty}g\left(m, n+\tfrac{1}{2},1-x\right)\right).
 \end{align*}
 \noindent
 To obtain the final result, let $x \to 0$ in the identity above. The use of the dominated convergence theorem, the evaluation
 \begin{equation*}
 \lim\limits_{x \to 0} \left( \frac{1}{2^{2n}} \frac{(2 + x - \sqrt{x(4+x)})^{2n+1}}{\sqrt{x(4+x)}} - \frac{1}{\sqrt{x}} \right) = 
 - 2n -1,
 \end{equation*}
 \noindent
 the parity of the Chebyshev polynomials and the special values 
 \begin{equation*}
 U_{2n}(0) = (-1)^{n}, \, U_{2n}(1) = 2n+1, \, U_{2n} \left( \tfrac{1}{2} \right) = \frac{2}{\sqrt{3}} \sin \left( \frac{(2n+1)\pi}{3} 
 \right),
 \end{equation*}
 \noindent
 complete the proof of  the stated formula.
 \end{proof}

\section{Asymptotics of Zagier polynomials}\label{asymptotics}
Corollary \ref{zagier-asymp} is proved here. The proof is given only for the even-index case, the one for odd being similar. 

First assume $x\neq \frac{1}{4}, \frac{3}{4}$. The idea is straightforward: we divide both sides of Theorem \ref{thm-main-0} by $(-1)^{n} \pi Y_{2n}(4 \pi) \cos( 2 \pi x)$ 
and show that the resulting right-hand side approaches $1$ as $n\to\infty$. However, in the case of the Bessel function series, justification of the interchange of the order 
of limit and summation is needed, and which is interesting in its own right. To that end, it is first shown that
\begin{equation}\label{series-limit}
\lim_{n\to\infty}\frac{1}{\left(-1\right)^{n}\pi Y_{2n}\left(4\pi\right)\cos\left(2\pi x\right)}\sum_{m=2}^{\infty}\left(-1\right)^{n}\pi Y_{2n}\left(4\pi m\right)\cos\left(2\pi mx\right)=0.
\end{equation}
Indeed, use \eqref{modification} to rewrite the above limit in the form
\begin{align*}
%&\sum_{m=2}^{\infty}\frac{\left(\left(-1\right)^{n}\pi Y_{2n}\left(4\pi m\right)+\frac{1}{2\sqrt{m}}\right)\cos\left(2\pi mx\right)}{\left(-1\right)^{n}\pi Y_{2n}\left(4\pi\right)\cos\left(2\pi x\right)}\nonumber\\
&\lim_{n\to\infty}\sum_{m=2}^{\infty}\frac{\left(\left(-1\right)^{n}\pi Y_{2n}\left(4\pi m\right)+\frac{1}{2\sqrt{m}}\right)\cos\left(2\pi mx\right)}{\left(-1\right)^{n}\pi Y_{2n}\left(4\pi\right)\cos\left(2\pi x\right)}\\
&\quad-\lim_{n\to\infty}\frac{1}{\left(-1\right)^{n}\pi Y_{2n}\left(4\pi\right)\cos\left(2\pi x\right)}\sum_{m=2}^{\infty}\frac{\cos\left(2\pi mx\right)}{2\sqrt{m}}.\nonumber\label{eq:1}
\end{align*}
The latter limit is equal to zero since \eqref{yorder-asymp} implies that
\begin{equation}\label{Y-denom}
\lim_{n\to\infty}\frac{(-1)^{n}}{Y_{2n}(4\pi)}=0,
\end{equation}
and also because the series $\sum_{m=2}^{\infty}\frac{\cos\left(2\pi mx\right)}{2\sqrt{m}}$ converges for $0<x<1$.

It is now shown that the order of limit and summation can be interchanged in the former limit. This requires the 
hypotheses of Lebesgue's dominated convergence theorem for series to hold, namely that
\begin{equation}\label{lebesgue-1}
\lim_{n\to\infty}\frac{\left(\left(-1\right)^{n}\pi Y_{2n}\left(4\pi m\right)+\frac{1}{2\sqrt{m}}\right)}{\left(-1\right)^{n}\pi Y_{2n}\left(4\pi\right)}=0,
\end{equation}
and that the sequence $\{s_n(m, x)\}_{n=1}^{\infty}$ , where
\begin{equation}\label{lebesgue-2}
s_n(m, x)=\frac{\left(\left(-1\right)^{n}\pi Y_{2n}\left(4\pi m\right)+\frac{1}{2\sqrt{m}}\right)\cos\left(2\pi mx\right)}{\left(-1\right)^{n}\pi Y_{2n}\left(4\pi\right)\cos\left(2\pi x\right)},
\end{equation}
is uniformly bounded in $n$. To prove \eqref{lebesgue-1}, note that \eqref{yorder-asymp} implies that
\begin{equation*}
\lim_{n\to\infty}\frac{Y_{2n}\left(4\pi m\right)}{Y_{2n}\left(4\pi\right)}=\lim_{n\to\infty}m^{-2n}=0
\end{equation*}
as $m\geq 2$. Together with \eqref{Y-denom}, this proves \eqref{lebesgue-1}.

In order to prove \eqref{lebesgue-2}, we first use the fact \cite[p.~927]{gradshteyn-2015a}, \cite[p.~446]{watson-1966a} 
that for $x>0$, the function
\begin{equation*}
x\mapsto x\left[J_{\nu}^{2}\left(x\right)+Y_{\nu}^{2}\left(x\right)\right]
\end{equation*}
decreases monotonically, if $\nu>\frac{1}{2}$. In particular, for $m, n\in\mathbb{N}$,
\[
4\pi m\left[J_{2n}^{2}\left(4\pi m\right)+Y_{2n}^{2}\left(4\pi m\right)\right]\le4\pi\left[J_{2n}^{2}\left(4\pi\right)+Y_{2n}^{2}\left(4\pi\right)\right].
\]
From \cite[p.~227, formula 10.14.1]{olver-2010a}, we have $|J_{\nu}\left(x\right)|\leq 1$ for $\nu\ge0,\thinspace\thinspace x\in\mathbb{R}$, so that
\begin{eqnarray}\label{inter-ineq}
Y_{2n}^{2}\left(4\pi m\right) & \le & J_{2n}^{2}\left(4\pi m\right)+Y_{2n}^{2}\left(4\pi m\right)\le\frac{1}{m}\left[J_{2n}^{2}\left(4\pi\right)+Y_{2n}^{2}\left(4\pi\right)\right]\\
 & \le & \frac{1}{m}\left[1+Y_{2n}^{2}\left(4\pi\right)\right]\nonumber.
\end{eqnarray}
Next, it is shown that for $n\in\mathbb{N}$,
\begin{equation}\label{abs-c}
\frac{1}{Y_{2n}^{2}\left(4\pi\right)}\le c
\end{equation}
for some absolute constant $c$. Note that the asymptotic formula \eqref{yorder-asymp} implies that 
the sequence $\left\{Y_{2n}^{-2}(4\pi)\right\}_{n=1}^{\infty}$ tends to $0$ as $n\to\infty$. 
So for $n$ sufficiently large, say $n\geq n_{0}$, the inequality $\frac{1}{Y_{2n}^{2}(4\pi)}\leq 1$ holds. 

If we can further show that $Y_{2n}(4\pi)\neq 0$ for any $n<n_{0}$, then \eqref{abs-c} will be 
proved as one can take $c$ to be $\text{max}\left(1,Y_{2}(4\pi), Y_{4}(4\pi), \cdots, Y_{2(n_{0}-1)}(4\pi)\right)$. 
To show that this is indeed true is the objective of the following lemma. We could not find a reference to it 
in the literature, and our proof of it is short and nice, hence given here.
\begin{lemma}\label{Y-nonzero}
For any positive integer $n$, $Y_{2n}(4\pi)\neq 0$.
\end{lemma}
\begin{proof}
Let $y_{\nu,k}$ denote the $k^{\textup{th}}$ zero of the Bessel function $Y_{\nu}\left(x\right)$. The last line on 
page $68$ in \cite{elbert-2001a} implies the inequality 
\begin{equation*}
y_{\nu,k}>\nu+k\pi-\frac{1}{2}\hspace{5mm}(\nu>\tfrac{1}{2}, k\in\mathbb{N}).
\end{equation*}
Thus the first zero of $Y_{2n}(x)$ satisfies $y_{2n,1}>4\pi$ if $2n+\pi-\frac{1}{2}>4\pi$ which happens
when $n>\frac{3\pi}{2}+\frac{1}{4}\approx 4.96239.$ Hence $Y_{2n}\left(4\pi\right)\ne0$
for $n\ge5.$ Also, it can be checked that $Y_{2n}\left(4\pi\right)\ne0$
for $1\le n\le4$, since $Y_{2}\left(4\pi\right)\approx 0.134559,$\,
$Y_{4}\left(4\pi\right)\approx -0.0357975,$ $Y_{6}\left(4\pi\right)\approx -0.14694$
and $Y_{8}\left(4\pi\right)\approx 0.246447$. This completes the proof of the lemma.
\end{proof}
\noindent
The result in Lemma \ref{Y-nonzero} along with the previous discussion now proves \eqref{abs-c}. 

Now divide both sides of \eqref{inter-ineq} by $Y_{2n}^{2}(4\pi)$ and use \eqref{abs-c} to obtain
\begin{equation*}
\frac{Y_{2n}(4\pi m)}{Y_{2n}(4\pi)}\leq \frac{\sqrt{c+1}}{\sqrt{m}},
\end{equation*}
thereby proving \eqref{lebesgue-2}.

Thus, from Lebesgue's dominated convergence theorem, it is seen that
{\allowdisplaybreaks\begin{align*}
&\lim_{n\to\infty}\sum_{m=2}^{\infty}\frac{\left(\left(-1\right)^{n}\pi Y_{2n}\left(4\pi m\right)+\frac{1}{2\sqrt{m}}\right)\cos\left(2\pi mx\right)}{\left(-1\right)^{n}\pi Y_{2n}\left(4\pi\right)\cos\left(2\pi x\right)}\\
&=\sum_{m=2}^{\infty}\lim_{n\to\infty}\frac{\left(\left(-1\right)^{n}\pi Y_{2n}\left(4\pi m\right)+\frac{1}{2\sqrt{m}}\right)\cos\left(2\pi mx\right)}{\left(-1\right)^{n}\pi Y_{2n}\left(4\pi\right)\cos\left(2\pi x\right)}\nonumber\\
&=0\nonumber,
\end{align*}}
as can be seen from \eqref{Y-denom} and \eqref{lebesgue-1}. This proves \eqref{series-limit}.

Our next task is to show that
\begin{equation}\label{g-limit}
\lim_{n\to\infty}\frac{1}{2^{2n+1}\left(-1\right)^{n}\pi Y_{2n}\left(4\pi\right)\cos\left(2\pi x\right)}\left(\sum_{m=1}^{\infty}g\left(m, n, x\right)+\sum_{m=1}^{\infty}g\left(m, n, 1-x\right)\right)=0.
\end{equation}
%For a fixed $x\in(0,1)$ and fixed $r>0$, the function $y\mapsto g(y, n, x)$, 
%defined in \eqref{gynx}, is decreasing on $[1,\infty)$, since
% {\allowdisplaybreaks\begin{align*}
%\frac{d}{dy}g\left(y, n, x\right)&=\frac{-\left(1+x+y-\sqrt{\left(y-1+x\right)\left(y+3+x\right)}\right)^{2n}}{\left(\left(y-1+x\right)\left(y+3+x\right)\right)^{\frac{3}{2}}}\\
%&\quad\times\left(1+x+y+2n\sqrt{\left(y-1+x\right)\left(y+3+x\right)}\right)\\
%&<0.
%\end{align*}}
Again using the series version of Lebesgue's dominated convergence theorem, we observe that as $n\to\infty$,
\begin{equation*}
\sum_{m=1}^{\infty}g\left(m, n, x\right)\sim g\left(1, n, x\right)=\frac{2^{4n}}{\left(2+x+\sqrt{x\left(x+4\right)}\right)^{2n}\sqrt{x\left(x+4\right)}},
\end{equation*}
so that by \eqref{yorder-asymp},
\begin{align*}
&\lim_{n\to\infty}\frac{\sum_{m=1}^{\infty}g\left(m, n, x\right)}{2^{2n+1}\left(-1\right)^{n}\pi Y_{2n}\left(4\pi\right)\cos\left(2\pi x\right)}\\
&=\frac{1}{\cos\left(2\pi x\right)\sqrt{\pi x\left(x+4\right)}}\lim_{n\to\infty}\frac{\left(-1\right)^{n+1}2^{2n-1}\sqrt{n}}{\left(2+x+\sqrt{x\left(x+4\right)}\right)^{2n}}\left(\frac{e\pi}{n}\right)^{2n}\\
&=0.
\end{align*}
Similarly,
\begin{equation*}
\lim_{n\to\infty}\frac{1}{2^{2n+1}\left(-1\right)^{n}\pi Y_{2n}\left(4\pi\right)\cos\left(2\pi x\right)}\sum_{m=1}^{\infty}g\left(m, n, 1-x\right)=0.
\end{equation*}
Thus \eqref{g-limit} is proved.
Finally the fact that
\begin{equation}\label{u-limit}
\lim_{n\to\infty}\frac{\left(U_{2n-1}\left(\frac{x+1}{2}\right)+U_{2n-1}\left(\frac{x}{2}\right)+U_{2n-1}\left(\frac{x-1}{2}\right)+U_{2n-1}\left(\frac{x-2}{2}\right)\right)}{4\left(-1\right)^{n}\pi Y_{2n}\left(4\pi\right)\cos\left(2\pi x\right)}=0
\end{equation}
easily follows from \eqref{chebyshev} and \eqref{yorder-asymp}. Thus \eqref{series-limit}, 
\eqref{g-limit} and \eqref{u-limit} along with Theorem \ref{thm-main-0} prove \eqref{zagier-even-asymp} for $x\neq\frac{1}{4}, \frac{3}{4}$. 

When $x=\frac{1}{4}, \frac{3}{4}$, the first term of the Bessel function series in Theorem \ref{thm-main-0} 
is zero. However, a logic exactly similar to the one above can be worked out starting with the second term of the series, namely, 
$(-1)^{n+1}\pi Y_{2n}(8\pi)$.

\section{Proof of the Zagier-type formula}\label{zagier-type}
This section is devoted to proving Theorem \ref{zagier-type-thm}. Let $x=1/2$ in Theorem \ref{thm-main-0}, and note that $U_{2n-1}(x)$ is an odd function of $x$. This gives
\begin{align}\label{bnshalf}
B_{2n}^{*}\left(\frac{1}{2}\right)=\sum_{m=1}^{\infty}(-1)^{m+n}\pi Y_{2n}(4\pi m)+\frac{1}{2^{2n}}\sum_{m=1}^{\infty}\frac{\left(m+\frac{3}{2}-\sqrt{\left(m-\frac{1}{2}\right)\left(m+\frac{7}{2}\right)}\right)^{2n}}{\sqrt{\left(m-\frac{1}{2}\right)\left(m+\frac{7}{2}\right)}}.
\end{align} 
Now let $x=-3/2$, $k=2$, and replace $n$ by $2n$ in \eqref{lem10.2}. Along with \eqref{bnshalf}, this gives
 \begin{align*}
B_{2n}^{*}\left(-\frac{3}{2}\right)
%&=\sum_{m=1}^{\infty}(-1)^{m+n}\pi Y_{2n}(4\pi m)-\frac{1}{2}\left(U_{2n-1}\left(\frac{1}{4}\right)+U_{2n-1}\left(\frac{3}{4}\right)\right)\nonumber\\
%&\quad+\frac{1}{2^{2n}}\sum_{m=1}^{\infty}\frac{\left(m+\frac{3}{2}-\sqrt{\left(m-\frac{1}{2}\right)\left(m+\frac{7}{2}\right)}\right)^{2n}}{\sqrt{\left(m-\frac{1}{2}\right)\left(m+\frac{7}{2}\right)}}\nonumber\\
&=\sum_{m=1}^{\infty}(-1)^{m}\left((-1)^n\pi Y_{2n}(4\pi m)+\frac{1}{2\sqrt{m}}\right)-\frac{\left(\sqrt{2}-1\right)}{2}\zeta\left(\frac{1}{2}\right)\nonumber\\
&\quad-\frac{1}{2}\left(U_{2n-1}\left(\frac{1}{4}\right)+U_{2n-1}\left(\frac{3}{4}\right)\right)\nonumber\\
&\quad+\frac{1}{2^{2n}}\sum_{m=1}^{\infty}\frac{\left(m+\frac{3}{2}-\sqrt{\left(m-\frac{1}{2}\right)\left(m+\frac{7}{2}\right)}\right)^{2n}}{\sqrt{\left(m-\frac{1}{2}\right)\left(m+\frac{7}{2}\right)}},
\end{align*}
where we also used the well-known identity
\begin{equation*}
\sum_{m=1}^{\infty}\frac{(-1)^{m+1}}{m^s}=\left(1-2^{1-s}\right)\zeta(s),
\end{equation*}
valid for Re$(s)>0$.

Add the corresponding sides of the above equation to those of Zagier's formula \eqref{zagier-sum} to obtain
 \begin{align}\label{add}
B_{2n}^{*}\left(-\frac{3}{2}\right)+B_{2n}^{*}&=\sum_{m=1}^{\infty}(1+(-1)^m)\left((-1)^{n}\pi Y_{2n}(4\pi m)+\frac{1}{2\sqrt{m}}\right)\nonumber\\
&\quad-n-\frac{1}{2}\left(U_{2n-1}\left(\frac{1}{4}\right)+U_{2n-1}\left(\frac{3}{4}\right)\right)-\frac{1}{\sqrt{2}}\zeta\left(\frac{1}{2}\right)\nonumber\\
&\quad+\frac{1}{2^{2n}}\bigg\{\sum_{m=1}^{\infty}\frac{\left(m+2-\sqrt{m(m+4)}\right)^{2n}}{\sqrt{m(m+4)}}\nonumber\\
&\quad\quad\quad\quad+\sum_{m=1}^{\infty}\frac{\left(m+\frac{3}{2}-\sqrt{\left(m-\frac{1}{2}\right)\left(m+\frac{7}{2}\right)}\right)^{2n}}{\sqrt{\left(m-\frac{1}{2}\right)\left(m+\frac{7}{2}\right)}}\bigg\}.
\end{align}
Note that
\begin{equation*}
\sum_{m=1}^{\infty}(1+(-1)^m)\left((-1)^{n}\pi Y_{2n}(4\pi m)+\frac{1}{2\sqrt{m}}\right)=2\sum_{m=1}^{\infty}\left((-1)^{n}\pi Y_{2n}(8\pi m)+\frac{1}{2\sqrt{2m}}\right), 
\end{equation*}
and that the sum of the two series in \eqref{add} can be written as
 {\allowdisplaybreaks\begin{align*}
&\sum_{m=1}^{\infty}\bigg\{\frac{\left(\frac{2m}{2}+2-\sqrt{\frac{2m}{2}(\frac{2m}{2}+4)}\right)^{2n}}{\sqrt{\frac{2m}{2}(\frac{2m}{2}+4)}}+\frac{\left(\frac{2m-1}{2}+2-\sqrt{\left(\frac{2m-1}{2}\right)\left(\frac{2m-1}{2}+4\right)}\right)^{2n}}{\sqrt{\left(\frac{2m-1}{2}\right)\left(\frac{2m-1}{2}+4\right)}}\bigg\}\nonumber\\
&=\frac{1}{2^{2n-1}}\sum_{m=1}^{\infty}\frac{\left(m+4-\sqrt{m(m+8)}\right)^{2n}}{\sqrt{m(m+8)}}.
\end{align*}}
This completes the proof.

\section{Some open problems}\label{op}
We conclude this paper by discussing three open problems.
\subsection{Proving Lemmas \ref{lemma-two4} and \ref{lemma-two5} through Poisson summation formula} We begin with a lemma.
\begin{lemma}\label{J-poisson}
For \textup{Re}$(\nu)>0$ and $0<x<1$, the identity
\begin{align*}
&\sum_{m=1}^{\infty}J_{\nu}(4\pi m)\cos(2\pi mx)\nonumber\\
&=\frac{\cos\left(\nu\arcsin\left(\frac{x}{2}\right)\right)}{\sqrt{(4\pi)^{2}-(2\pi x)^{2}}}+\frac{\cos\left(\nu\arcsin\left(\frac{x+1}{2}\right)\right)}{\sqrt{(4\pi)^{2}-(2\pi (x+1))^{2}}}\nonumber\\
&\quad+\frac{\cos\left(\nu\arcsin\left(\frac{1-x}{2}\right)\right)}{\sqrt{(4\pi)^{2}-(2\pi (1-x))^{2}}}+\frac{\cos\left(\nu\arcsin\left(\frac{2-x}{2}\right)\right)}{\sqrt{(4\pi)^{2}-(2\pi (2-x))^{2}}}\nonumber\\
&\quad-\sum_{m=2}^{\infty}\frac{(4\pi)^{\nu}\sin\left(\frac{\nu\pi}{2}\right)}{\sqrt{(2\pi (m+x))^{2}-(4\pi)^{2}}\left(2\pi (m+x)+\sqrt{(2\pi (m+x))^{2}-(4\pi)^{2}}\right)^{\nu}}\nonumber\\
&\quad-\sum_{m=3}^{\infty}\frac{(4\pi)^{\nu}\sin\left(\frac{\nu\pi}{2}\right)}{\sqrt{(2\pi (m-x))^{2}-(4\pi)^{2}}\left(2\pi (m-x)+\sqrt{(2\pi (m-x))^{2}-(4\pi)^{2}}\right)^{\nu}}
\end{align*}
holds.
\end{lemma}
%\textbf{Remark.} Linton \cite[[Equation (52)]{linton-2006a} has already worked out such a representation for $\nu=j, j\in\mathbb{N}$ in a form suitable for his work on diffraction theory.
\begin{proof}
Apply the Poisson summation formula \eqref{poisson} with $f(t)=J_{\nu}(4\pi |t|)\cos(2\pi x|t|)$, Re $\nu>0$, and use 
the integral evaluation \cite[p.~717, 6.671.2]{gradshteyn-2015a} \footnote{The condition that $\alpha$ 
and $\beta$ be positive is missing in the reference. But it can be found, for example, in \cite[p.~398, Section 13.4]{watson-1966a}.}
\begin{align*}
\int_{0}^{\infty}J_{\nu}(\alpha t)\cos(\beta t)\, dt=
\begin{cases}
\frac{\cos\left(\nu\arcsin\frac{\beta}{\alpha}\right)}{\sqrt{\alpha^{2}-\beta^{2}}},\hspace{3mm} (0<\beta<\alpha),\\
\frac{-\alpha^{\nu}\sin\left(\frac{\nu\pi}{2}\right)}{\sqrt{\beta^{2}-\alpha^{2}}\left(\beta+\sqrt{\beta^{2}-\alpha^{2}}\right)^{\nu}},\hspace{3mm} (0<\alpha<\beta).
\end{cases}
\end{align*}
\end{proof}

\subsubsection{An alternative approach to proving Lemma \ref{lemma-two4}}

Consider the double series
\begin{equation}\label{ixn}
\mathfrak{I}(x, n):=\sum_{k=1}^{\infty}\sum_{m=1}^{\infty}\frac{1}{k}\left(J_{2n+2k}(4\pi m)-J_{2n-2k}(4\pi m)\right)\cos(2\pi mx).
\end{equation}
We wish to use Lemma \ref{J-poisson} twice to simplify the result.  Observe that it requires 
Re$(\nu)>0$, so this requires to treat the case $k=n$ separately. In this situation, entry $8.522.1$ 
in \cite{gradshteyn-2015a} states that
\begin{multline}\label{j0}
    \sum_{m=1}^{\infty} J_{0}(m u ) \cos( m u v) \\ 
		=- \frac{1}{2} + 
    \sum_{\ell=1}^{j} \frac{1}{\sqrt{u^{2} - (2 \pi \ell + uv)^{2}}} + \frac{1}{u \sqrt{1-v^{2}}} + 
    \sum_{\ell=1}^{d} \frac{1}{\sqrt{u^{2} - (2 \pi \ell - uv)^{2}}},
    \end{multline}
    \noindent
    where $u >0, \, 0 \leq v < 1, \, 2 \pi j < u(1-v) < 2(j+1)\pi, \, 2 d \pi < u(1+v) < 2(d+1) \pi$ and $j+1, \, d+1 \in \mathbb{N}$. 
Now use Lemma \ref{J-poisson} twice in \eqref{ixn}, and also use \eqref{j0} to obtain after some simplification
\begin{align*}
\mathfrak{I}(x, n)&=\frac{1}{2n}-\frac{1}{\pi}\sum_{k=1}^{\infty}\frac{1}{k}\bigg\{\frac{\sin\left(2n\sin^{-1}\left(\frac{x}{2}\right)\right)\sin\left(2k\sin^{-1}\left(\frac{x}{2}\right)\right)}{\sqrt{4-x^2}}\nonumber\\
&\quad\quad\quad\quad\quad\quad\quad\quad+\frac{\sin\left(2n\sin^{-1}\left(\frac{x+1}{2}\right)\right)\sin\left(2k\sin^{-1}\left(\frac{x+1}{2}\right)\right)}{\sqrt{4-(x+1)^2}}\nonumber\\
&\quad\quad\quad\quad\quad\quad\quad\quad+\frac{\sin\left(2n\sin^{-1}\left(\frac{1-x}{2}\right)\right)\sin\left(2k\sin^{-1}\left(\frac{1-x}{2}\right)\right)}{\sqrt{4-(1-x)^2}}\nonumber\\
&\quad\quad\quad\quad\quad\quad\quad\quad+\frac{\sin\left(2n\sin^{-1}\left(\frac{2-x}{2}\right)\right)\sin\left(2k\sin^{-1}\left(\frac{2-x}{2}\right)\right)}{\sqrt{4-(2-x)^2}}\bigg\}.
\end{align*}
The definition in \eqref{chebyshev} and the classical formula \cite[p. 46, formula 1.441.1]{gradshteyn-2015a}
      \begin{equation*}
      \sum_{j=1}^{\infty} \frac{\sin j \theta}{j} = \frac{\pi - \theta}{2}, \quad \text{ for } 0 < \theta < 2 \pi
      \end{equation*}
      \noindent
			now yield
\begin{align*}
\mathfrak{I}(x, n)&=\frac{1}{2n}+\frac{(-1)^n}{2\pi}\bigg\{\cos^{-1}\left(\frac{x}{2}\right)U_{2n-1}\left(\frac{x}{2}\right)+\cos^{-1}\left(\frac{x+1}{2}\right)U_{2n-1}\left(\frac{x+1}{2}\right)\nonumber\\
&\quad\quad+\cos^{-1}\left(\frac{1-x}{2}\right)U_{2n-1}\left(\frac{1-x}{2}\right)+\cos^{-1}\left(\frac{2-x}{2}\right)U_{2n-1}\left(\frac{2-x}{2}\right)\bigg\}.
\end{align*}
Comparing with the result of Lemma \ref{lemma-two4} and the identity \eqref{form-PQ1}, it is easily seen that Lemma \ref{lemma-two4} 
can be proved this way if the following problem, which we leave for the interested reader, can be solved.

\textbf{Problem 1.} Prove that for $0<x<1$ and $n\in\mathbb{N}$,
\begin{align*}
&\sum_{k=1}^{\infty}\sum_{m=1}^{\infty}\frac{1}{k}\left(J_{2n+2k}(4\pi m)-J_{2n-2k}(4\pi m)\right)\cos(2\pi mx)\\
&=\sum_{m=1}^{\infty}\sum_{k=1}^{\infty}\frac{1}{k}\left(J_{2n+2k}(4\pi m)-J_{2n-2k}(4\pi m)\right)\cos(2\pi mx).
\end{align*}

\subsubsection{An alternative approach to proving Lemma \ref{lemma-two5}}
Using Lemma \ref{J-poisson}, it can be seen, after some simplification, that
\begin{align*}
&\lim_{\nu\to 2n}\frac{\partial}{\partial\nu}\sum_{m=1}^{\infty}J_{\nu}(4\pi m)\cos(2\pi mx)\\
 & =\frac{(-1)^{n}}{4 \pi} \bigg\{ \sin^{-1} \left( \frac{x}{2} \right) U_{2n-1} \left( \frac{x}{2} \right) + 
   \sin^{-1} \left( \frac{x+1}{2} \right) U_{2n-1} \left( \frac{x+1}{2} \right)  \\
   & \quad \quad \quad \quad\quad + \sin^{-1} \left( \frac{1-x}{2} \right) U_{2n-1} \left( \frac{1-x}{2} \right) + 
    \sin^{-1} \left( \frac{2-x}{2} \right) U_{2n-1} \left( \frac{2-x}{2}  \right) \bigg\} \\
    &\quad+\frac{(-1)^{n+1}}{4^{n+1}} \bigg(\sum_{m=1}^{\infty}g(m, n, x)+\sum_{m=1}^{\infty}g(m, n, 1-x)\bigg).
    \end{align*}
		If we compare this result with that in Lemma \ref{lemma-two5}, it is easily seen that one can prove the lemma 
		in this alternative way provided the following interchange of the order of summation can be proved.

\textbf{Problem 2.} Prove that for Re$(\nu)>0$, $0<x<1$ and $n\in\mathbb{N}$,
\begin{align*}
\sum_{m=1}^{\infty} \left.\frac{\partial }{\partial \nu}J_{\nu}(4 \pi m)\right|_{\nu=2n} \cos(2 \pi m x)=\lim_{\nu\to 2n}\frac{\partial}{\partial\nu}\sum_{m=1}^{\infty}J_{\nu}(4\pi m)\cos(2\pi mx).
\end{align*}
This, too, is left as an open problem for the reader to prove.

\textbf{Remark.} Problems analogous to above can be formulated for the odd-indexed case. 

\subsection{A relation between $B_{n}^{*}\left(-\frac{3}{2}\right)$ and $B_{n}^{*}$?}

Theorem \ref{zagier-type-thm} arose while trying to find a relation between 
$B_{n}^{*}\left(-\frac{3}{2}\right)$ and $B_{n}^{*}$ similar to the relation \cite[p.~4]{temme-1996a} 
\begin{equation*}
B_{n}\left(\frac{1}{2}\right)=\left(2^{1-n}-1\right)B_n
\end{equation*}
that exists for Bernoulli polynomials. Note that while the result $B_{n}(1-x)=(-1)^nB_{n}(x)$ shows symmetry 
along $x=1/2$, it was established in \cite[Theorem 11.1]{dixit-2014a} that
\begin{equation}\label{refsym}
B_{n}^{*}(-x-3)=(-1)^nB_{n}^{*}(x),
\end{equation} 
so that the symmetry for Zagier polynomials is along $x=-3/2$. Since the theory of Zagier polynomials nicely parallels that of
the Bernoulli polynomials, it is reasonable to look for a relation between $B_{n}^{*}\left(-\frac{3}{2}\right)$ 
and $B_{n}^{*}$. Of course, it is clear from \eqref{refsym} that $B_{2n+1}^{*}\left(-\frac{3}{2}\right)=0$. On 
the other hand, the sequence $\{B_{2n+1}^{*}\}$ is $6$-periodic and takes the values 
$\{ \tfrac{3}{4}, - \tfrac{1}{4}, \, - \tfrac{1}{4}, \, 
\tfrac{1}{4}, \, \tfrac{1}{4}, \, - \tfrac{3}{4}  \}$, readily implying
\begin{equation*}
B_{2n+1}^{*}\left(-\frac{3}{2}\right)=0\cdot B_{2n+1}^{*}.
\end{equation*}
Hence only the case of even indices is of real interest. We have not been able to find such a relation yet, if at all it exists.

\medskip

\no
\textbf{Acknowledgments}. 
The authors wish to thank Karl Mahlburg of Louisiana State University for interesting 
discussions on this subject. 
The third author acknowledges the 
partial support of NSF-DMS 1112656. The first author was a post-doctoral 
fellow funded in part by the same grant.  The work of the fourth author was partially supported by the iCODE Institute,
Research Project of the Idex Paris-Saclay.


\begin{thebibliography}{00}

\bibitem{apostol-1998a}
Apostol, T.M.: Introduction to Analytic Number Theory, Springer-Verlag, New York (1998).

\bibitem{asatryan}
Asatryan, A.A.: Summation of a {S}chl\"omilch type series, Proc. R. Soc. A \textbf{471}, no. 2183, 20150359 (2015).

\bibitem{coates}
Coates, C.V.: Bessel's functions of the second order, Quart. J. \textbf{XX}, 250--260 (1885).

\bibitem{dixit-2014b}
Coffey, M., De Angelis, V., Dixit, A., Moll, V.H., Straub, A., and Vignat, C., The Zagier modification of Bernoulli polynomials. Part II: Arithmetic properties of denominators,  Ramanujan J. \textbf{35}, 361--390 (2014).

\bibitem{cohen-2008a}
Cohen, H.: Number Theory: Volume II: Analytic and Modern Tools, Vol. 240. Springer (2008).

\bibitem{dilcher-1987a}
Dilcher, K.: Asymptotic behavior of Bernoulli, Euler and generalized Bernoulli polynomials, J.~Approx.~Theory~\textbf{49}, 321--330 (1987).

\bibitem{dixit-2014a}
Dixit, A., Moll, V.H., and Vignat, C.: The Zagier modification of Bernoulli numbers and a polynomial extension. Part I, Ramanujan J. \textbf{33}, 379--422 (2014).

\bibitem{elbert-2001a}
Elbert, \'{A}.: Some recent results on the zeros of Bessel functions and orthogonal polynomials, J. Comput. Appl. Math. \textbf{133}, 65--83 (2001).

\bibitem{gradshteyn-2015a}
Gradshteyn, I.S. and Ryzhik, I.M.: Table of Integrals, Series, and Products, Edited by D. Zwillinger and V. H.~Moll, Academic Press, New York, 8th
  edition (2015).

\bibitem{hansen-1975a}
Hansen, E.R.: A Table of Series and Products, Prentice-Hall series in automatic computation, Englewood-Cliffs: Prentice-Hall (1975).

\bibitem{hargreaves-1918a}
Hargreaves, R.: A diffraction problem, and an asymptotic theorem in Bessel's series, Phil. Mag.~\textbf{36}, 191--199 (1918).
 
%\bibitem{ifantis-1985a}
%E. K.~Ifantis and P. D.~Siafarikas.
%\newblock A differential equation for the zeros of {B}essel functions.
%\newblock {\em Appl. Anal.}, 20:269--281, 1985.

\bibitem{ignatowsky-1915a}
Ignatowsky, W.v.: \"{U}ber Reihen mit Zylinderfunktionen nach dem Vielfachen des Argumentes, Arch. d. Math. u. Phys.~\textbf{23}, 193--219 (1915).

\bibitem{jackson-1904a}
Jackson, W.H.: On the diffraction of light produced by an opaque prism of finite angle, Proc. London Math. Soc.~\textbf{2(1)}, 393--414 (1904).

\bibitem{linton-2006b}
Linton, C.M., Schl\"{o}milch series that arise in diffraction theory and their efficient computation, Technical Report, available at
\url{http://homepages.lboro.ac.uk/~macml1/schlomilch-techreport.pdf}

\bibitem{linton-2006a}
Linton, C.M.: Schl\"{o}milch series that arise in diffraction theory and their efficient computation, J. Phys. A: Math. Gen.~\textbf{39}, 3325--3339 (2006).

\bibitem{luke-1969b}
Luke, Y.L.: The Special Functions and Their Approximations, volume~2 of Mathematics in Science and Engineering, Academic Press, New York-London (1969).

\bibitem{macdonald-2013a}
Macdonald, H.M.: Electric waves, Cambridge University Press (2013).

\bibitem{mcphedran-1983a}
McPhedran, R.C.: A note on wedge function and echelette gratings, J. Mod. Opt.~\textbf{30(8)}, 1029--1034 (1983).

\bibitem{neilsen-1904a}  
Nielsen, N.: Handbuch der Theorie der Cylinderfunktionen, BG Teubner (1904).

\bibitem{olver-2010a}
Olver, F.W.J., Lozier, D.W., Boisvert, R.F., and Clark, C.W. (eds.), NIST Handbook of Mathematical Functions, Cambridge University Press (2010).

\bibitem{popov-1936a}
Popov, A.I.: Some remarks on Bessel functions, Trans. Leningrad Industr. Inst., Sect. Phys. Math.~\textbf{10}, 49--52 (1936).

\bibitem{temme-1996a}
Temme, N.M.: Special Functions. An introduction to the Classical Functions of Mathematical Physics, John Wiley and Sons, New York (1996).

\bibitem{titchmarsh-1948a}
Titchmarsh, E.C.: Theory of Fourier Integrals, 2nd ed., Clarendon Press, Oxford (1948).

\bibitem{twersky-1956a}
Twersky, V.: On the scatttering of waves by an infinite grating, Antennas and Propagation, IRE Transactions on~\textbf{4(3)}, 330--345 (1956).

\bibitem{twersky-1959b}
Twersky, V.: Scattering by quasi-periodic and quasi-random distributions, Antennas and Propagation, IRE Transactions on~\textbf{7(5)}, 307--319 (1959).

\bibitem{twersky-1961c}
Twersky, V.: Elementary function representations of Schl\"{o}milch series, Arch. Rational Mech. Anal.~\textbf{8}, 323--332 (1961).

\bibitem{watson-1966a}
Watson, G.N.: A treatise on the Theory of Bessel Functions, Cambridge University Press (1966).

\bibitem{whipple-1917a}
Whipple, F.W.J.: Diffraction by a wedge and kindred problems, Proc. London Math. Soc.~\textbf{XVI(2)}, 94--111 (1917).

\bibitem{zagier-1990b}
Zagier, D.: Hecke operators and periods of modular forms, Israel Math. Conf. Proc.~\textbf{3}, 321--336 (1990).

\bibitem{zagier-1998a}
Zagier, D.: A modified Bernoulli number, Nieuw Arch. Wisk.~\textbf{16}, 63--72 (1998).
%Nieuw {A}rchief voor {W}iskunde
\end{thebibliography}
\end{document}